\documentclass[12pt]{article}
\usepackage{amsmath, amssymb,amsfonts,bbm,verbatim,multirow,color}
\usepackage{booktabs}
\usepackage[sf,bf,SF,footnotesize]{subfigure}
\usepackage{graphicx}
\usepackage{amsthm}
\usepackage{fullpage}
\usepackage{amscd}
\usepackage{epsfig}
\usepackage{natbib}
\usepackage{enumerate}
\usepackage{array}
\usepackage[small]{caption}
\RequirePackage[colorlinks,citecolor=blue,urlcolor=blue,breaklinks]{hyperref}

\newtheorem{thm}{Theorem}

\newtheorem{prop}[thm]{Proposition}
\newtheorem{cor}[thm]{Corollary}

\newcommand{\E}{\mathbb{E}}

\newcommand{\footremember}[2]{%
    \footnote{#2}
    \newcounter{#1}
    \setcounter{#1}{\value{footnote}}%
}
\newcommand{\footrecall}[1]{%
    \footnotemark[\value{#1}]%
} 

\def\hat{\widehat}

\begin{document}

\title{Locally private non-asymptotic testing of discrete distributions is faster using interactive mechanisms}

\author{Thomas B. Berrett\footremember{fund1}{Financial support from the French National Research Agency (ANR) under the grant Labex Ecodec (ANR-11-LABEX-0047)}\footremember{fund2}{Financial support from the French National Research Agency (ANR) under the grant ANR-17-CE40-0003 HIDITSA} \,  and Cristina Butucea\footrecall{fund1} \\
  CREST, ENSAE, Institut Polytechnique de Paris
}


\maketitle

\begin{abstract}
We find separation rates for testing multinomial or more general discrete distributions under the constraint of $\alpha$-local differential privacy. We construct efficient randomized algorithms and test procedures, in both the case where only non-interactive privacy mechanisms are allowed and also in the case where all sequentially interactive privacy mechanisms are allowed. The separation rates are faster in the latter case. We prove general information theoretical bounds that allow us to establish the optimality of our algorithms among all pairs of privacy mechanisms and test procedures, in most usual cases. Considered examples include testing uniform, polynomially and exponentially decreasing distributions.
\end{abstract}
\section{Introduction}

Hypothesis testing of discrete distributions is intensively used as a first step in data based decision making and it is now also a component of many machine learning algorithms. Given samples from an unknown probability distribution $p$ and a known reference distribution $p_0$, the goal of a goodness-of-fit test is to decide whether $p$ fits $p_0$, or is signficantly different to it in some suitable sense. Here we will measure distance between distributions using either the $\mathbb{L}_1$ norm or the $\mathbb{L}_2$ norm, with our alternative hypotheses consisting of all distributions $p$ whose distance from $p_0$ is above a certain threshold.
Our goal is to make accurate decisions, i.e. with low error probabilities, for distributions $p$ as close to $p_0$ as possible. The smallest separation between $p_0$ and the alternative hypothesis for which it remains possible to reliably distinguish between the two hypotheses is known as the uniform separation rate, $\delta$. 
Its optimality is proven by showing that, whenever $p$ is closer to $p_0$ than $\delta$, no test procedure will be able to distinguish them with small error probabilities. As shown by \citet{VV2017}, the dependence of $\delta$ on $p_0$ in the standard problem without privacy constraints is pronounced and intricate.

In this work, we quantify how the constraint of local differential privacy affects the optimal separation rate. Differential privacy \cite{DMNS2006} is the most popular formalism under which the analyst statistically randomizes data to be published in order to protect the privacy of the individuals in the study. The way in which the data is randomized, known as the privacy mechanism, must be carefully chosen to preserve the information in the data that is most pertinent for the task at hand, though it is well-established that the cost of protecting privacy is necessarily a deterioration in statistical performance. Local differential privacy, in which there is no trusted curator who has access to all the original data, is more stringent than the original differential privacy constraint, and it is often observed in the literature that in this context we experience a further deterioration of the achievable performance of estimation and test procedures, which are allowed to use private data only.

\subsection{Our contributions}

We find the optimal separation rate around an arbitrary discrete distribution $p_0$ under local differential privacy constraints and this optimality involves two steps. On the one hand, we provide efficient and statistically optimal pairs of privacy mechanisms and their associated test procedures whose error probabilities are small for all distributions $p$ further from $p_0$ than $\delta$. On the other hand, we show that whenever $p$ is closer to $p_0$ than the separation rate $\delta$, no pair of privacy mechanism and test procedure is able to distinguish $p$ and $p_0$. We show that faster rates are attainable using interactive privacy mechanisms (which are allowed to use one original sample together with all the available private samples) than using non-interactive privacy mechanisms (which are only allowed to one original sample at a time). The interactive mechanism that we use is a two-step procedure: for the first half of the sample we employ a Laplace mechanism and estimate the unknown probabilities, for the second half we randomize encoding this information and build a $\chi^2$-type statistic. Our optimality results show that the separation rates are optimal in most usual cases. Let us stress the fact that the second part of Theorem~\ref{Thm:InfoBounds} is particularly useful in the case of noninteractive mechanisms where, as it was also noted by \cite{LLL2020},
the usual inequalities in  \cite{DJW2018} can only result in suboptimal lower bounds. 
We highlight the examples of (nearly) uniform distributions, and distributions with polynomially or exponentially decreasing tails. 
All results are valid nonasymptotically, that is with a finite number of samples.

\subsection{Related work}

The study of discrete data under local privacy constraints can be traced back at least as far as \citet{Warner1965}, in which the classical randomised response mechanism was introduced to provide privacy when estimating the proportion of a population that belongs to a certain group. This problem can be thought of as a special case, with an alphabet of size two, of the problem of estimating the probability vector of a multinomial random variable. For a general (finite) alphabet size, \citet{DJW2013} derive upper and lower bounds on the minimax estimation risk in this more general problem for both the $\mathbb{L}_1$ and $\mathbb{L}_2$ metrics, and, in particular, show that a generalization of the randomized response algorithm is rate-optimal in certain regimes. Besides the standard estimation problem, one can also consider the problem of estimating succinct histograms, or heavy hitters; see, for example, \citet{BS2015}.

Compared with estimation problems, hypothesis testing is relatively under-explored in the setting of local differential privacy. Early work includes \citet{KOV2014,KOV2016}, in which the aim is to test the simple hypotheses $H_0: P=P_0$ vs. $H_1: P=P_1$, where $P_0$ and $P_1$ are two given discrete distributions. Under a mutual information-based local privacy constraint, \citet{LSTd2017} considered the more general problem of $m$-ary hypothesis testing. The goodness-of-fit testing problem, which we consider in this paper, was investigated in \citet{GR2017,Sheffet2018}, where the authors provide analyses of procedures based on chi-squared tests and the optimal non-private test of \citet{VV2017}, respectively, under specific privacy mechanisms. Unfortunately, these privacy mechanisms and tests are typically suboptimal. As we show here, a corrected chi-squared statistic, calculated using suitably generated private data, may be optimal for testing uniformity, but for general null hypotheses a more subtle procedure is required. \citet{ACFT2019} studies the problem of testing uniformity, and provides tests that they show are optimal among all tests using their chosen privacy mechanisms. Techniques for proving more general lower bounds, over general classes of privacy mechanisms, are developed in \citet{ACT2018} and applied to uniformity testing. The role of interactivity in locally private testing is studied in \citet{JMNR2019}, where it is shown that an optimal procedure for testing the hypotheses $H_i: P \in \mathcal{P}_i, i=1,2$, for disjoint, convex $\mathcal{P}_i$, is non-interactive. This is in contrast to our results, which show that, in goodness-of-fit testing, interactive procedures achieve significantly faster separation rates.

In the non-private setting, goodness-of-fit testing of discrete distributions has recently received a great deal of attention. \citet{VV2017} found near optimal separation rates that show that the difficulty of the problem depends intricately on the specific null hypothesis; see also \citet{DK2016} and the survey article \citet{BW2017}. The problem has also been considered in the non-local differentially private setting \citep{WLK2015,GLRV2016,CDK2017,ADR2018,ASZ2018}. Upper bounds have been provided, which have been shown to be nearly optimal in certain regimes for the case of a uniform null. Besides goodness-of-fit, there is also work in this setting on testing independence between discrete variables \citep{WLK2015,GLRV2016}, and in testing independence between a discrete and a continuous variable by constructing differentially private versions of classical rank-based tests \citep{CKSBG2019}.


\section{Preliminaries}

Let $\mathcal{P}_d = \{ p=(p(1),\ldots,p(d)) \in [0,1]^d : \sum_{j=1}^d p(j) =1 \}$ denote the set of all probability vectors in $d$-dimensions. For $x=(x(1),\ldots,x(d)) \in \mathbb{R}^d$ write $\|x\|_1= \sum_{j=1}^d |x(j)|$ for the $\ell_1$-norm and $\|x\|_2$ for the Euclidean norm of $x$. For $p \in \mathcal{P}_d$ we say that $X$ is distributed according the probability $p$, $X \sim p$, if $\mathbb{P}(X=j)=p(j)$ for each $j=1,\ldots,d$. Given $p_0 \in \mathcal{P}_d$, $\delta >0$ and data $X_1,\ldots,X_n \overset{\mathrm{i.i.d}}{\sim} p$ we will study the $ \mathbb{L}_1$ and $ \mathbb{L}_2$ problems of testing the hypotheses:
\begin{equation}
\label{Eq:Hypotheses}
	H_0: p=p_0 \quad \text{vs.} \quad H_1(\delta, \mathbb{L}_r): \left\{ p \in \mathcal{P}_d \text{ such that } \|p-p_0\|_r \geq \delta \right\},
\end{equation}
for $r$ equal to 1 and 2 respectively, under an $\alpha$-local differential privacy (LDP) constraint on the allowable tests. An $\alpha$-LDP privacy mechanism $Q$ generates private data $Z_i$ taking values in $\mathcal{Z}$ via the conditional distribution $Q_i(\cdot|x_i,z_1,...,z_{i-1})$ such that
 \[
 \sup_A \sup_{z_1,...,z_{i-1}} \sup_{x,x'} \frac{Q_i(A|x,z_1,...,z_{i-1})}{Q_i(A|x',z_1,...,z_{i-1})} \leq e^\alpha, \text{ for all } i=1,...,n.
 \]
 Given an $\alpha$-LDP privacy mechanism $Q$, let $\Phi_Q = \{\phi: \mathcal{Z}^n \rightarrow [0,1] \}$ denote the set of all (randomized) tests of the hypotheses~\eqref{Eq:Hypotheses} based on $Z_1,\ldots,Z_n$. We can define the minimax testing risk when using the privacy mechanism $Q$ by
\[
	\mathcal{R}_n(p_0,\delta,Q, \mathbb{L}_r) := \inf_{\phi \in \Phi_Q} \sup_{p \in H_1( \delta, \mathbb{L}_r )} \bigl\{ \mathbb{E}_{p_0} (\phi) + \mathbb{E}_p(1-\phi) \bigr\}.
\]
Then, writing $\mathcal{Q}_\alpha$ for the collection of all $\alpha$-LDP sequentially-interactive privacy mechanisms, we can define the $\alpha$-LDP minimax testing risk of~\eqref{Eq:Hypotheses} by
\[
	\mathcal{R}_{n,\alpha}(p_0,\delta, \mathbb{L}_r) := \inf_{Q \in \mathcal{Q}_\alpha} \mathcal{R}_n(p_0,\delta,Q, \mathbb{L}_r).
\]
Given $\gamma \in (0,1)$ we aim to find the $\alpha$-LDP minimax testing radius defined by
\[
	\mathcal{E}_{n,\alpha}(p_0, \mathbb{L}_r) := \inf \{ \delta >0 : \mathcal{R}_{n,\alpha}(p_0,\delta, \mathbb{L}_r) \leq \gamma \}
\]
Moreover, we will also aim to find the minimax testing risk of~\eqref{Eq:Hypotheses} under the additional restriction that $Q$ is a non-interactive $\alpha$-LDP privacy mechanism. Letting $\mathcal{Q}_\alpha^\mathrm{NI} \subset \mathcal{Q}_\alpha$ denote the subset of all such privacy mechanisms, we similary define
\[
	\mathcal{R}_{n,\alpha}^\mathrm{NI}(p_0,\delta, \mathbb{L}_r) := \inf_{Q \in \mathcal{Q}_\alpha^\mathrm{NI}} \mathcal{R}_n(p_0,\delta,Q, \mathbb{L}_r), \quad \mathcal{E}_{n,\alpha}^\mathrm{NI}(p_0, \mathbb{L}_r) := \inf \{ \delta >0 : \mathcal{R}_{n,\alpha}^\mathrm{NI}(p_0,\delta, \mathbb{L}_r) \leq \gamma \}.
\]
Note that this formalism is equivalent to finding the smallest sample size $n$ in order to attain a given accuracy, i.e. testing risk measure.

%

\section{Building private samples and optimal test procedures}

We split the support of the multinomial distribution $p_0$ into a main set $B$ and a tail set $ B^c$. As is now typical in such problems, we combine a $\chi^2$ test on $B$ (which is not the usual corrected $\chi^2$ statistic) with a tail-test on $B^c$ in order to achieve optimality.

These tests procedures use privatized data. Our non-interactive privacy mechanisms use classical Laplace randomization -- see, for example, \cite{DJW2018}. The interactive procedure involved in the $\chi^2$-test is novel in the context of discrete distributions; see \cite{BRS2020} for a similar mechanism in the continuous setting. It is a two-step procedure, that uses part of the sample in order to estimate the frequencies $\hat p_j$ and then randomizes the other part of the sample using a censored value of $\hat p_j - p_0(j)$. A simple average of this second part of the private sample allows us to construct the $\chi^2$-test. Thus the latter procedure encodes partial information on the distribution in the randomization of the second part of the sample and benefits from it.

\subsection{Non-interactive privacy mechanisms}

Assume that the sample size is even and that the data is given by $X_1,\ldots,X_{2n}$. Given a nonempty subset $B \subseteq [d]$ we define a non-interactive privacy mechanism $Q_B \in \mathcal{Q}_\alpha^\mathrm{NI}$ and a test $\phi_B \in \Phi_{Q_B}$ as follows. Given an i.i.d. sequence $(W_{ij})_{i\in[n],j\in B}$ of $\mathrm{Laplace}(1)$ random variables, for each $i\in[n]$ and $j \in B$ write
\begin{equation}
\label{Eq:BasicPrivacyMechanism}
	Z_{ij} := \mathbbm{1}_{\{X_i=j\}} + \frac{2}{\alpha} W_{ij}.
\end{equation}
We have that $(Z_{ij})_{i\in[n],j\in B}$ is an $\alpha$-LDP version of $X_1,\ldots,X_n$ \citep[see, e.g.,][]{GR2017}. Now write the $\chi^2$ test statistic
\[
	S_B:= \sum_{j \in B} \frac{1}{n(n-1)} \sum_{i_1 \neq i_2} \{Z_{i_1j} - p_0(j)\}\{ Z_{i_2 j} - p_0(j) \}.
\]
Letting $(W_i)_{i=n+1}^{2n}$ denote a second sequence of i.i.d. $\mathrm{Laplace}(1)$ random variables, for $i=n+1,\ldots,2n$ we set
\[
	Z_i = \mathbbm{1}_{\{ X_i \in B^c\}} + \frac{2}{\alpha} W_i.
\]
Then again $(Z_i)_{i=n+1}^{2n}$ is an $\alpha$-LDP version of $X_{n+1},\ldots,X_{2n}$ version of $X_{n+1},\ldots,X_{2n}$. Further, define
\[
	T_B:= \frac{1}{n} \sum_{i=n+1}^{2n} \{Z_i - p_0(B^c)\},
\]
for the tail test statistic, where we write $p_0(B^c) = \sum_{j \in B^c} p_0(j)$. With the critical values $C_{1,B}:=\{ 656 |B|/ (n(n-1) \alpha^4 \gamma ) \}^{1/2} $ and $C_{2,B}:= 6 /(n \alpha^2 \gamma)^{1/2} $, we finally set
\[
	\phi_B(Z_1,\ldots,Z_{2n}):= \left\{ \begin{array}{cc} 1 & \text{if } S_B \geq C_{1,B} \text{ and/or } T_B \geq C_{2,B} \\ 0 & \text{otherwise} \end{array} \right. ,
\]
that is, we reject $H_0$ if either $S_B \geq C_{1,B}$ or $T_B \geq C_{2,B}$.

\begin{thm}
\label{Thm:NonIntUpperBound}
When $\alpha \in (0,1]$, for any $\emptyset \neq B \subseteq [d]$ we have that
\[
	\mathcal{E}^\mathrm{NI}_{n,\alpha}(p_0, \mathbb{L}_1)  \leq 8 \max\biggl[ 12 \Bigl\{ \frac{|B|^3}{n(n-1) \alpha^4 \gamma^2} \Bigr\}^{1/4}, \, p_0(B^c) \biggr]. 
\]
and
\[
	\mathcal{E}^\mathrm{NI}_{n,\alpha}(p_0, \mathbb{L}_2)  \leq 8 \max\biggl[ 12 \Bigl\{ \frac{|B|}{n(n-1) \alpha^4 \gamma^2} \Bigr\}^{1/4}, \, p_0(B^c) \biggr]. 
\]
\end{thm}
Note that we can actually include discrete distributions on all of $\mathbb{N}$. We prove the tightest upper bounds by finding the sets $B$ that minimize the right-hand sides in Theorem~\ref{Thm:NonIntUpperBound}. The search algorithm is trivial if we order the sequence $p_0(\cdot)$ in decreasing order. Indeed, then it is straightforward to see that the optimal $B$ is of the form $\{1,\ldots,j\}$ in both cases, with the first term in the maximum increasing with $j$, and $p_0(B^C)$ decreasing with $j$. Therefore, there are always finite sets (possibly large) that minimize the right-hand sides.

Theorem~\ref{Thm:NonIntUpperBound} yields the following immediate corollary.

\begin{cor}
\label{Cor:UpperBound}
Let
\begin{align*}
	j_* &= j_*(n \alpha^2 , p_0, \mathbb{L}_1) := \min \biggl\{ j=1,\ldots,d : \frac{j^{3/4}}{(n \alpha^2)^{1/2}} \geq \sum_{j'=j+1}^d p_0(j') \biggr\}\\
	j_{**} &= j_{**}(n \alpha^2 , p_0, \mathbb{L}_2) := \min \biggl\{ j=1,\ldots,d : \frac{j^{1/4}}{(n \alpha^2)^{1/2}} \geq \sum_{j'=j+1}^d p_0(j') \biggr\}.
\end{align*}
When $\alpha \in (0,1]$, there exist $C_1=C_1(\gamma)$ and $C_2=C_2(\gamma)$ such that
\[
	\mathcal{E}^\mathrm{NI}_{n,\alpha}(p_0, \mathbb{L}_1) \leq C_1 \frac{j_*^{3/4}}{(n \alpha^2)^{1/2}} \quad \text{and} \quad
	\mathcal{E}^\mathrm{NI}_{n,\alpha}(p_0, \mathbb{L}_2)  \leq C_2 \frac{j_{**}^{1/4}}{(n \alpha^2)^{1/2}}.
\]
\end{cor}
In particular, for testing the uniform distribution over $[d]$, this corollary shows in both cases a loss of a factor ${d}^{1/4}$ with respect to the minimax rates that we can attain without privacy.

In Corollary~\ref{Cor:UpperBound} we always have $j_*, \, j_{**} \leq d$, so we can always say that $\mathcal{E}^\mathrm{NI}_{n,\alpha}(p_0,\mathbb{L}_1) \lesssim d^{3/4}/(n \alpha^2)^{1/2}$ and that $\mathcal{E}^\mathrm{NI}_{n,\alpha}(p_0,\mathbb{L}_2) \lesssim d^{1/4}/(n \alpha^2)^{1/2}$. However, for some values of $p_0$ our upper bound is better than this.

It is important to note here that the $\mathbb{L}_1$ test behaves very differently in this context from the case of non private setup. It is known since \cite{VV2017}, see also \cite{BW2017}, that in the direct setup a weighted $\chi^2$-test is needed in order to attain the optimal rates. This is due to the heteroscedasticity of the multinomial model (the variances of the counts are proportional to their probabilities) and a correction for very small variances needs to be included. Unlike this setup, the privacy constraint induces an unavoidable homoscedastic term in the variance of the $\chi^2$-square test and makes the correction useless in this case, resulting in a loss in the rate. We will see in Section~\ref{Sec:LowerBounds} that these rates for the $\mathbb{L}_1$ problem are essentially optimal.

The $\mathbb{L}_2$ test also combines the $\chi^2$ and the tail tests in order to achieve nearly optimal rates and this is also in contrast with the non-private case where the $\chi^2$ test is sufficient. However, as we will describe in Section~\ref{Subsec:Examples}, for polynomially decreasing distributions there is a gap between our non-interactive upper and lower bounds in some cases. Nevertheless, our results in Sections~\ref{Sec:IntUpperBounds} and~\ref{Sec:LowerBounds} do demonstrate a significant gap between non-interactive and interactive rates, even in these settings.


\subsection{Interactive privacy mechanisms and faster rates}
\label{Sec:IntUpperBounds}

Assume here that the sample is split in 3 parts, or that the data is given by $X_1,\ldots,X_{3n}$. The data $X_1,...,X_{2n}$ is used to build the interactive test statistic $D_B$ as described hereafter, while the third part of the sample, $X_{2n+1},...,X_{3n}$, is used to build the same test statistic $T_B$ as in the noninteractive setup.

We define an interactive privacy mechanism $Q_\mathrm{I} \in \mathcal{Q}_\alpha$ and a test $\psi_B \in \Phi_{Q_\mathrm{I}}$ as follows. With the first half of the sample, as in~\eqref{Eq:BasicPrivacyMechanism} with $B=[d]$, generate an i.i.d. sequence $(W_{ij})_{i\in[n],j\in [d]}$ of $\mathrm{Laplace}(1)$ random variables, and for each $i\in[n]$ and $j \in [d]$ write
\[
	Z_{ij} := \mathbbm{1}_{\{X_i=j\}} + \frac{2}{\alpha} W_{ij}.
\]
We again have that $(Z_{ij})_{i\in[n],j\in [d]}$ is an $\alpha$-LDP version of $X_1,\ldots,X_n$. For each $j \in [d]$ set
\[
	\hat{p}_j:= \frac{1}{n} \sum_{i=1}^n Z_{ij}.
\]
Set $c_\alpha = \frac{e^\alpha+1}{e^\alpha-1}$ and $\tau = (n \alpha^2)^{-1/2}$. As for the second half of the sample, for each $i =n+1,\ldots,2n$, generate $Z_i$ taking values in $\{-c_\alpha \cdot \tau,c_\alpha \cdot \tau \}$ such that
\[
	\mathbb{P}(Z_i = c_\alpha \cdot \tau | X_i =j) = \frac{1}{2} \Bigl( 1 + \frac{[\hat{p}_j - p_0(j)]^\tau_{-\tau}}{c_\alpha \cdot \tau} \Bigr),
\]
where we denote by
$$
[v]_{-\tau}^\tau = (-\tau) \vee v \wedge \tau, \quad \text{for all } v \in \mathbb{R}
$$
the censoring operator.
Then $(Z_i)_{i=n+1,\ldots,2n}$ is an $\alpha$-LDP version of $X_{n+1},\ldots,X_{2n}$ \citep{BRS2020}. We then define the test statistic
\[
	D_n = \frac{1}{n} \sum_{i=n+1}^{2n} Z_i - \sum_{j =1}^d p_0(j) \{[\hat{p}_j - p_0(j)]^\tau_{-\tau} \}
\]
and $C_3:= \frac{e+1}{e-1}  \frac{(4/\gamma)^{1/2}}{n \alpha^2}$.
The final test is
\[
	\psi_B(Z_1,\ldots,Z_{3n}):= \left\{ \begin{array}{cc} 1 & \text{if } D_n \geq C_{3} \text{ and/or } T_B \geq C_{2,B} \\ 0 & \text{otherwise} \end{array} \right.,
\]
that is, we reject $H_0$ if either $D_n \geq C_{3}$ or $T_B \geq C_{2,B}$.

\begin{thm}
\label{Thm:IntUpperBound}
There exists a universal constant $C$ such that when $\alpha \in (0,1]$, for any $\emptyset \neq B \subseteq [d]$, we have that
\begin{align*}
	\mathcal{E}_{n,\alpha}(p_0, \mathbb{L}_1)  \leq C \max \biggl\{ \frac{|B|^{1/2}}{(n \alpha^2 \gamma^2)^{1/2}}, p_0(B^c) \biggr\} \quad \text{and} \quad \mathcal{E}_{n,\alpha}(p_0, \mathbb{L}_2)  \leq \frac{C}{(n \alpha^2 \gamma^2)^{1/2}}.
\end{align*}
In particular, let
\begin{align*}
	\tilde j &= \tilde j(n \alpha^2 , p_0, \mathbb{L}_1) := \min \biggl\{ j=1,\ldots,d : \frac{j^{1/2}}{(n \alpha^2)^{1/2}} \geq \sum_{j'=j+1}^d p_0(j') \biggr\}.
\end{align*}
Then there exists $C_1=C_1(\gamma)$ such that, when $\alpha \in (0,1]$, we have
\[
	\mathcal{E}_{n,\alpha}(p_0, \mathbb{L}_1) \leq C_1 \frac{\tilde j^{1/2}}{(n \alpha^2)^{1/2}}.
\]
\end{thm}


\section{Non-asymptotic optimality}
\label{Sec:LowerBounds}

Attaining the rates through a particular randomization of the original sample and an associated test scheme does not prevent us from trying to improve on these choices. Instead, our lower bound results show that there are no better choices of privacy mechanisms and test procedures that would improve the test risk (or the separation rate) uniformly over the set of discrete distributions. It is of particular interest to show that no other $\alpha$-LDP Markov kernels could be combined with any of the tests to improve on the upper bounds of our rates. There are however multiple choices of such couples leading to the optimal rates that we have described.

Proving the optimality of our methods consists of building a family $\{p_\xi: \xi \in \mathcal{V} \}$ that belongs to the alternative set of probability distributions $H_1(\delta)$ with high probability and then reducing the test problem to testing between $p_0$ under the null and the mixture of the $p_\xi$ under the alternative.
\begin{prop}\label{Prop:LowerBound} If $\{p_\xi: \xi \in \mathcal{V} \}$ is a family of distributions such that
\[
P_\xi (p_\xi \not \in H_1(\delta)) \leq \gamma_1, \quad \text{for some } \gamma_1>0.
\]
then, for arbitrary $\eta$ in (0,1), we have
\[
\mathcal{R}_{n,\alpha}(p_0,\delta) \geq \inf_{Q \in \mathcal{Q}_\alpha} (1- \eta) \left( 1 - \frac 1\eta TV(QP_0^n, E_\xi QP_\xi^n ) \right) - \gamma_1.
\]
\end{prop}
It is sufficient to show that $TV(QP_0^n, E_\xi QP_\xi^n )  \leq \eta \cdot \gamma_2$ such that $(1-\eta)(1-\gamma_2)-\gamma_1 \geq \gamma$. Standard inequalities
prove that it is sufficient to bound from above the Kullback--Leibler or the $\chi^2$ discrepancy between the private distribution under the null and the average of conveniently chosen private distributions under the set of alternatives.

The way the previous discrepancies relate to the underlying distributions of the data proves to be significantly different in the cases when we are constrained to use non-interactive privacy mechanisms only, and when we are allowed to use any privacy mechanism.

\bigskip

\noindent {\bf Information theoretical bounds for testing.} For maximal generality, we assume that the privacy mechanisms may act differently on each sample $X_i$. An interactive procedure acts through $q_i(z_i|X_i=j, z_1,...,z_{i-1})$ on $X_i$ and the resulting $Z_i$ is distributed, conditionally on $Z_1,...,Z_{i-1}$, according to $m_i^\xi (z_i|Z_1,...,Z_{i-1}) = p_\xi^\top q_i(z_i|\cdot , Z_1,...,Z_{i-1})$. A non-interactive procedure acts simply through $q_i(z_i|X_i=j)$ on $X_i$ and the resulting $Z_i$ is distributed according to $m_i^\xi(z_i) = p_\xi^T q_i(z_i | \cdot)$.

\begin{thm} \label{Thm:InfoBounds} Given the previous family of distributions $\{p_\xi: \xi \in \mathcal{V}\}$, we have
\[
KL(QP_0^n, E_\xi QP_\xi^n) \leq E_\xi \left[(p_\xi - p_0)^\top \Omega (p_\xi - p_0) \right],
\]
where the matrix $\Omega$ has elements
\[
\Omega_{j,k} \! = \!\! \sum_{i=1}^n \E_{p_0} \!\! \int \!\! \left(\frac{q_i(z_i|j,Z_1,...,Z_{i-1})}{m_i^0(z_i|Z_1,...,Z_{i-1})}-1 \right) \!\! \left(\frac{q_i(z_i|k,Z_1,...,Z_{i-1})}{m_i^0(z_i|Z_1,...,Z_{i-1})}-1 \right) \! m_i^0(z_i|Z_1,...,Z_{i-1}) dz_i.
\]
In the particular case of non-interactive privacy mechanisms, we have for independent copies $\xi,\, \xi'$
\[
\chi^2(QP_0^n, E_\xi QP_\xi^n) \leq E_{\xi, \xi'} \left[\exp\left( (p_\xi - p_0)^\top \Omega (p_{\xi'} - p_0)\right) \right] - 1,
\]
where $\Omega$ takes the simpler form
\[
\Omega_{j,k} = \sum_{i=1}^n \int \left(\frac{q_i(z_i|j)}{m_i^0(z_i)}-1 \right) \left(\frac{q_i(z_i|k)}{m_i^0(z_i )}-1 \right) m_i^0(z_i ) dz_i.
\]
\end{thm}

\subsection{Non-interactive approach}

Recall that $\mathcal{E}_{n,\alpha}^\mathrm{NI}(p_0)$ is the $\alpha$-LDP minimax testing radius when we restrict to non-interactive privacy mechanisms. 
We have the following result.

\begin{thm}
\label{Thm:NonIntLowerBound}
There exist $c_1=c_1(\gamma)>0$ and $c_2=c_2(\gamma)>0$ such that for all $\alpha \in (0,1]$ we have
\[
	\mathcal{E}_{n,\alpha}^\mathrm{NI}(p_0, \mathbb{L}_1) \geq c_1 \max_{j=1,\ldots,d} \min \biggl\{ \frac{j^{3/4}}{(n \alpha^2)^{1/2}},\, \frac{jp_0(j)}{ \log^{1/2}(2 j)} \biggr\}
\]
and
\[
	\mathcal{E}_{n,\alpha}^\mathrm{NI}(p_0, \mathbb{L}_2) \geq c_2 \max_{j=1,\ldots,d} \min \biggl\{ \frac{j^{1/4}}{(n \alpha^2)^{1/2}},\, \frac{j^{1/2} p_0(j)}{ \log^{1/2}(2 j)} \biggr\}.
\]
\end{thm}

We have the following immediate corollary.
\begin{cor}
\label{Cor:LowerBound}
Let
\begin{align*}
	\ell_* &= \ell_*(n\alpha^2,p_0,\mathbb{L}_1) := \max \biggl\{ j =1,\ldots,d : \frac{j^{3/4}}{(n \alpha^2)^{1/2}} \leq \frac{j p_0(j)}{\log^{1/2}(2j)} \biggr\} \\
	\ell_{**} &= \ell_{**}(n\alpha^2,p_0, \mathbb{L}_2) := \max \biggl\{ j =1,\ldots,d : \frac{j^{1/4}}{(n \alpha^2)^{1/2}} \leq \frac{j^{1/2} p_0(j)}{\log^{1/2}(2j)} \biggr\}.
\end{align*}
Then there exist $c_1=c_1(\gamma)>0$ and $c_2=c_2(\gamma)>0$ such that when $\alpha \in (0,1]$ we have
\[
	\mathcal{E}_{n,\alpha}^\mathrm{NI}(p_0, \mathbb{L}_1) \geq c_1 \frac{\ell_*^{3/4}}{(n \alpha^2)^{1/2}}
\quad \text{and} \quad
	\mathcal{E}_{n,\alpha}^\mathrm{NI}(p_0, \mathbb{L}_2) \geq c_2  \frac{\ell_{**}^{1/4}}{(n \alpha^2)^{1/2}}.
\]
\end{cor}

According to the behaviour of $p_0$, we may have identical or different values for $\ell_*$ and $\ell_{**}$.

In many examples of interest, these lower bounds match our previous upper bounds in Corollary~\ref{Cor:UpperBound} up to log factor, even though $\ell_*$ and $\ell_{**}$ do not solve exactly the same problems as $j_*$ and $j_{**}$, respectively.

\subsection{Interactive approach}

Under their most general form the privacy mechanisms we allow are sequentially interactive. 
As shown by Theorem~\ref{Thm:IntUpperBound} and Corollary~\ref{Cor:LowerBound}, the optimal rates for testing are faster with interactive procedures than with non-interactive procedures. The following theorem shows that the upper bounds in Theorem~\ref{Thm:IntUpperBound} are optimal.

\begin{thm}
\label{Thm:IntLowerBound}
There exist $c_1=c_1(\gamma)>0$ and $c_2=c_2(\gamma)>0$ such that when $\alpha \in (0,1]$ we have
\[
	\mathcal{E}_{n,\alpha}(p_0, \mathbb{L}_1) \geq c_1 \max_{j=1,\ldots,d} \min \biggl\{ \frac{j^{1/2}}{(n \alpha^2)^{1/2}},\, \frac{p_0(j)}{ \log^{1/2}(2 j)} \biggr\}
\]
and
\[
	\mathcal{E}_{n,\alpha}(p_0, \mathbb{L}_2) \geq c_2  \frac{1}{(n \alpha^2)^{1/2}}.
\]
In particular, let
\begin{align*}
	\tilde{\ell} &= \tilde{\ell}(n\alpha^2,p_0,\mathbb{L}_1) := \max \biggl\{ j =1,\ldots,d : \frac{j^{1/2}}{(n \alpha^2)^{1/2}} \leq \frac{ p_0(j)}{\log^{1/2}(2j)} \biggr\} .
\end{align*}
Then there exist $c_1=c_1(\gamma)>0$ such that when $\alpha \in (0,1]$ we have
\[
 \mathcal{E}_{n,\alpha}(p_0, \mathbb{L}_1) \geq c_1 \frac{\tilde{\ell}^{1/2}}{(n \alpha^2)^{1/2}} .
\]
\end{thm}

\section{Particular classes of distributions}

In this section we explicitly calculate the separation rates in several examples. See Section~\ref{Subsec:Examples} for more detailed and more general calculations. Inequalities $\gtrsim$ are valid up to $\log$ factors.

\bigskip

\noindent {\bf Nearly uniform distributions} Suppose that $p_0(j) \propto j^{-\beta}$ for some $\beta \in [0,1)$, and that $d^{3/4}/(n \alpha^2)^{1/2} \leq (1-\beta)/(\log^{1/2}(2d))$. Then
\begin{align*}
	\frac{d p_0(d)}{\log^{1/2}(2d)} = \frac{d^{1-\beta}}{\log^{1/2}(2d) \sum_{\ell=1}^d \ell^{-\beta}} \geq \frac{d^{1-\beta}}{\log^{1/2}(2d) \int_0^d x^{-\beta} \,dx} = \frac{1-\beta}{\log^{1/2}(2d)} \geq \frac{d^{3/4}}{(n\alpha^2)^{1/2}},
\end{align*}
and it follows that $\ell_* = d$. Thus, in this setting, $\mathcal{E}_{n,\alpha}^\mathrm{NI}(p_0, \mathbb{L}_1) \gtrsim d^{3/4} / (n\alpha^2)^{1/2} $.
Concerning the $\mathbb{L}_2$ rates, $\mathcal{E}_{n,\alpha}^\mathrm{NI}(p_0, \mathbb{L}_2) \gtrsim d^{1/4} / (n\alpha^2)^{1/2}$ if $\beta \leq 1/4$, whereas it is $\gtrsim d^{1/4} / (n\alpha^2)^{1/2} \wedge d^{1/2 - \beta} \wedge  (n \alpha^2)^{- (\beta - 1/2)/(2\beta - 1/2)}$ if $\beta >1/4$.

\noindent {\bf Polynomially decreasing distributions} Suppose that $p_0(j) \propto j^{-1-\beta}$ for some $\beta >0$, as for example is the case for the Pareto distributions used in extreme value theory. It is shown in Section~\ref{Subsec:Examples} that, when $1 \leq n \alpha^2 \leq (d/C)^{2\beta+3/2}$ we have that $j_* \leq \lceil C (n \alpha^2)^{1/(2\beta+3/2)} \rceil $. On the other hand, if $n\alpha^2 > (d/C)^{2\beta+3/2}$ then we will just say that $j_* \leq d$. It follows that
\[
	\mathcal{E}^\mathrm{NI}_{n,\alpha}(p_0, \mathbb{L}_1) \lesssim \frac{j_*^{3/4}}{(n\alpha^2)^{1/2}} \lesssim \min \Bigl\{ (n \alpha^2)^{-\frac{2 \beta}{4 \beta+3}}, \frac{d^{3/4}}{(n \alpha^2)^{1/2}} \Bigr\}.
\]
From the corresponding lower bounds, we get
\begin{eqnarray*}
	\mathcal{E}_{n,\alpha}^\mathrm{NI}(p_0, \mathbb{L}_1)& \gtrsim &  \bigl\{ n \alpha^2 \log^{3/(4 \beta)}(n \alpha^2) \bigr\}^{-2\beta/(4 \beta+3)} \wedge \frac{d^{3/4}}{(n \alpha^2)^{1/2}}.
\end{eqnarray*}
So, the lower bounds match the upper bounds up to $\log$ factors in this case.

\bigskip

\noindent {\bf Exponentially decreasing distributions} Suppose that $p_0(j) \propto \exp(-j^\beta)$ for some $\beta >0$. More generally, we may include the geometric distribution with $p_0(j) \propto p^j = \exp(-j \log(1/p))$ or $p_0(j) \propto j^\eta \exp(-cj^\beta)$, for $\eta$ real number and $c,\, \beta >0$.
The upper bounds match the lower bounds in this case, and lead e.g. in the case of noninteractive privacy mechanisms and $\mathbb{L}_1$ norm to the rate
\[
	\mathcal{E}_{n,\alpha}^\mathrm{NI}(p_0 , \mathbb{L}_1) \asymp \min \biggl\{ \frac{ \log^{3/(4 \beta)}(n \alpha^2)}{\sqrt{n \alpha^2}}, \frac{d^{3/4}}{\sqrt{n \alpha^2}} \biggr\}.
\]
Analogous calculations can be done for interactive mechanisms and $\mathbb{L}_2$ norm.
Table~\ref{SRtable} summarizes the minimax separation rates, for examples of distrbution probabilities $p_0$. They are optimal up to $\log$ factors except for the $\mathbb{L}_2$ distance in the case of uniform and polynomially decreasing distributions.

\begin{table}
  \caption{ Separation rates for testing discrete distributions }
  \label{SRtable}
  \centering
  \begin{tabular}{lcccc}
    \toprule
    &\multicolumn{2}{c}{Noninteractive} & \multicolumn{2}{c}{Interactive} \\
    \midrule
    $p_0$   & $\mathbb{L}_1$  & $\mathbb{L}_2$ & $\mathbb{L}_1$  & $\mathbb{L}_2$ \\
    \midrule
   Uniform$[d]$  & $\frac{d^{3/4}}{\sqrt{n\alpha^2}}$
   & \begin{tabular}{l}
   $\leq \frac{d^{1/4}}{\sqrt{n\alpha^2}}$\\
   $\gtrsim \frac{d^{1/4}}{\sqrt{n\alpha^2}} \wedge \frac 1{\sqrt{d}} $
   \end{tabular}
   & $\frac{d^{1/2}}{\sqrt{n\alpha^2}}$ & $\frac{1}{\sqrt{n\alpha^2}}$     \\
   &&&&\\
   $\propto j^{-1-\beta}$ & $(n \alpha^2)^{-\frac{2 \beta}{4 \beta +3}} \wedge \frac{d^{3/4}}{\sqrt{n\alpha^2}}$ &
   \begin{tabular}{l}
   $\leq (n \alpha^2)^{-\frac{2 \beta}{4 \beta +1}} \wedge \frac{d^{1/4}}{\sqrt{n\alpha^2}}$\\
   $\gtrsim(n \alpha^2)^{-\frac{2 \beta+1}{4 \beta +3}} \wedge \frac{d^{1/4}}{\sqrt{n\alpha^2}}$
   \end{tabular}  & $(n \alpha^2)^{-\frac{2 \beta}{4 \beta +2}} \wedge \frac{d^{1/2}}{\sqrt{n\alpha^2}}$ & $\frac{1}{\sqrt{n\alpha^2}}$     \\
   &&&&\\
   $\propto j^{\eta} e^{-c j^\beta}$  & $\frac{\log^{3/(4\beta)}(n \alpha^2) \wedge d^{3/4}}{\sqrt{n\alpha^2}}$ & $\frac{\log^{1/(4\beta)}(n \alpha^2) \wedge d^{1/4}}{\sqrt{n\alpha^2}}$  & $\frac{\log^{2/(4\beta)}(n \alpha^2) \wedge d^{1/2}}{\sqrt{n\alpha^2}}$ & $\frac{1}{\sqrt{n\alpha^2}}$     \\
    \bottomrule
  \end{tabular}
\end{table}

\section{Proofs of main theorems}

\begin{proof}[Proof of Theorem~\ref{Thm:NonIntUpperBound}]

We first calculate means and variances of our two test statistics, starting with the $U$-statistic $S_B$. Define the function $h : \mathbb{R}^B \times \mathbb{R}^B \rightarrow \mathbb{R}$ by
\[
		h(z_1,z_2) = \sum_{j \in B} \{z_{1j} - p_0(j)\}\{z_{2j} - p_0(j)\}
\]
so that $S_B = \frac{1}{n(n-1)} \sum_{i_1 \neq i_2} h(Z_{i_1},Z_{i_2})$. It is clear that
\[
	\mathbb{E} S_B = \sum_{j \in B} \{p(j) - p_0(j)\}^2.
\]
Now, define
\begin{align*}
	\zeta_1 := \mathrm{Var} \bigl( \mathbb{E} \{ h(Z_1,Z_2) | Z_1 \} \bigr) \quad \text{and} \quad \zeta_2 := \mathrm{Var} \bigl( h(Z_1,Z_2) \bigr).
\end{align*}
Using \citet[Lemma A, p.183]{Serfling1980} and the fact that $\mathrm{Cov}(Z_{1j},Z_{1j'}) = \mathbbm{1}_{\{j=j'\}}\{p(j) + 8/\alpha^2\} - p(j)p(j')$, we have that
\begin{align*}
	\binom{n}{2}& \mathrm{Var} \, S_B = \sum_{c=1}^2 \binom{2}{c} \binom{n-2}{2-c} \zeta_c =  (2n-3) \zeta_1 + (\zeta_2-\zeta_1) \\
	&= (2n-3) \mathrm{Var} \biggl( \sum_{j \in B} \{p(j) - p_0(j)\} \{Z_{2j} - p_0(j)\} \biggr) \\
	& \hspace{150pt} + \mathbb{E} \biggl\{ \mathrm{Var} \biggl( \sum_{j \in B} \{Z_{1j}-p_0(j)\}\{Z_{2j} - p_0(j)\} \biggm| Z_1 \biggr) \biggr\} \\
	& = 2(n-1) \sum_{j,j' \in B} \{p(j) - p_0(j)\}\{p(j')-p_0(j')\} \mathrm{Cov} (Z_{1j}, Z_{1j'})  + \sum_{j,j' \in B} \mathrm{Cov} (Z_{1j},Z_{1j'} )^2 \\
	& = 2(n-1) \sum_{j \in B} \{p(j) + 8/\alpha^2\} \{p(j) - p_0(j) \}^2 -2(n-1) \biggl( \sum_{j \in B} p(j) \{p(j) - p_0(j)\} \biggr)^2 \\
	& \hspace{50pt} + \sum_{j \in B} p(j)^2 \{1-2p(j)\} + \biggl( \sum_{j \in B} p(j) \biggr)^2 + \frac{64}{\alpha^4} |B| + \frac{16}{\alpha^2} \sum_{j \in B} p(j) \{1-p(j)\} \\
	& \leq \frac{18(n-1)}{\alpha^2} \sum_{j \in B} \{p(j) - p_0(j)\}^2 + \frac{82 |B|}{\alpha^4}.
\end{align*}
As a result,
\[
	\mathrm{Var} \, S_B \leq \frac{36}{n \alpha^2} \sum_{j \in B} \{p(j) - p_0(j)\}^2 + \frac{164 |B|}{n(n-1)\alpha^4}.
\]
We now turn to the test statistic $T_B$. First, it is clear that
\[
	\mathbb{E} T_B = p(B^c) - p_0(B^c).
\]
Moreover,
\begin{align*}
	\mathrm{Var} \, T_B = \frac{1}{n} \Bigl( \mathrm{Var} \, \mathbbm{1}_{\{X_{n+1} \in B^c\}}  + \frac{4}{\alpha^2} \mathrm{Var} \, W_{n+1} \Bigr) = \frac{1}{n} \Bigl[ p(B)\{1-p(B)\} + \frac{8}{\alpha^2} \Bigr] \leq \frac{9}{n \alpha^2}.
\end{align*}

Now, under $H_0$ we have that
\begin{align*}
\mathbb{P}(\phi_B =1 ) &\leq \mathbb{P} ( S_B \geq C_{1,B} ) + \mathbb{P}(T_B \geq C_{2,B} ) \\
	& \leq \frac{n(n-1) \alpha^4\gamma}{656|B|} \times \frac{164 |B|}{n(n-1) \alpha^4} + \frac{n \alpha^2 \gamma}{36} \times \frac{9}{n \alpha^2} = \frac{\gamma}{2}.
\end{align*}
Now suppose that we have
\begin{align}
\label{Eq:DeltaLowerBound}
	\delta &\geq 8 \max\biggl[ 12 \Bigl\{ \frac{|B|^3}{n(n-1) \alpha^4 \gamma^2} \Bigr\}^{1/4}, p_0(B^c) \biggr],
\end{align}
which implies
\begin{align*}
\delta 	& \geq 2 \max \biggl[ 24 \Bigl\{ \frac{|B|^3}{n(n-1) \alpha^4 \gamma^2} \Bigr\}^{1/4}, 2 p_0(B^c) + \frac{6 + 3 \sqrt{2}}{(n \alpha^2 \gamma)^{1/2}} \biggr] .
\end{align*}
Then, under $H_1 ( \delta, \mathbb{L}_1 ) $, at least one of
\begin{equation}
\label{Eq:Case1}
	\sum_{j \in B} |p(j) - p_0(j)| \geq  24 \Bigl\{ \frac{|B|^3}{n(n-1) \alpha^4 \gamma^2} \Bigr\}^{1/4}
\end{equation}
or
\begin{equation}
\label{Eq:Case2}
	\sum_{j \in B^c} |p(j) - p_0(j)| \geq 2p_0(B^c) + \frac{6 + 3 \sqrt{2}}{(n \alpha^2 \gamma)^{1/2}}
\end{equation}
must hold. If~\eqref{Eq:Case1} holds then we have that
\begin{align*}
	\mathbb{P} & ( S_B <  C_{1,B} ) \leq \frac{\mathrm{Var} \, S_B}{ [ \mathbb{E} S_B -  C_{1,B} ]^2} \leq \frac{\frac{36}{n \alpha^2}\sum_{j \in B} \{p(j) - p_0(j)\}^2 + \frac{164 |B|}{n(n-1)\alpha^4} }{[ \sum_{j \in B} \{p(j) - p_0(j)\}^2 -  \{ \frac{656 |B|}{n(n-1) \alpha^4 \gamma} \}^{1/2} ]^2} \\
	& \leq \frac{\frac{36}{n \alpha^2}\sum_{j \in B} \{p(j) - p_0(j)\}^2}{[\sum_{j \in B} \{p(j) - p_0(j)\}^2-  \{ \frac{656 |B|}{n(n-1) \alpha^4 \gamma} \}^{1/2} ]^2} + \frac{\frac{164 |B|}{n(n-1)\alpha^4} }{[ 576 \{ \frac{|B|}{n(n-1) \alpha^4 \gamma} \}^{1/2} -  \{ \frac{656 |B|}{n(n-1) \alpha^4 \gamma} \}^{1/2} ]^2} \\
	& \leq \frac{144}{n \alpha^2 \sum_{j \in B} \{p(j) - p_0(j)\}^2} + \frac{756 \gamma}{ 576^2} \leq \frac{144 \gamma}{576} + \frac{756 \gamma}{576^2} < \frac{\gamma}{2}.
\end{align*}
On the other hand, if~\eqref{Eq:Case2} holds then we have that $\mathbb{E}T_B = p(B^c) - p_0(B^c) \geq \frac{6 + 3 \sqrt{2}}{(n \alpha^2 \gamma)^{1/2}}$ and hence
\begin{align*}
	\mathbb{P}( T_B < C_{2,B} ) \leq \frac{\mathrm{Var} \, T_B}{ \{ \mathbb{E} T_B - \frac{6}{(n \alpha^2 \gamma)^{1/2}} \}^2} \leq \frac{n \alpha^2 \gamma}{18} \times \frac{9}{n \alpha^2} = \frac{\gamma}{2}.
\end{align*}
In conclusion, whenever $H_1 (\delta, \mathbb{L}_1 )$ holds and $\delta$ satisfies the lower bound in~\eqref{Eq:DeltaLowerBound}, we have that $\mathbb{P}(\phi_B =0) \leq \gamma/2$, and the result follows.

Under $H_1(\delta, \mathbb{L}_2)$ and using $\sqrt{a+b} \leq \sqrt{a}+ \sqrt{b}$:
$$
(\sum_{j \in B} |p(j) - p_0(j)|^2)^{1/2} + \sum_{j \in B^c} |p(j) - p_0(j)| \geq \|p-p_0\|_2\geq \delta.
$$
Now, we suppose that we have instead of (\ref{Eq:DeltaLowerBound}):
\begin{align*}
	\delta &\geq 8 \max\biggl[ 12 \Bigl\{ \frac{|B|}{n(n-1) \alpha^4 \gamma^2} \Bigr\}^{1/4}, p_0(B^c) \biggr] \\
	& \geq 2 \max \biggl[ 24 \Bigl\{ \frac{|B|}{n(n-1) \alpha^4 \gamma^2} \Bigr\}^{1/4}, 2 p_0(B^c) + \frac{6 + 3 \sqrt{2}}{(n \alpha^2 \gamma)^{1/2}} \biggr] \nonumber.
\end{align*}
That implies, at least one of
\begin{equation*}
	(\sum_{j \in B} |p(j) - p_0(j)|^2)^{1/2} \geq  24 \Bigl\{ \frac{|B|}{n(n-1) \alpha^4 \gamma^2} \Bigr\}^{1/4}
\end{equation*}
or
\begin{equation*}
	\sum_{j \in B^c} |p(j) - p_0(j)| \geq 2p_0(B^c) + \frac{6 + 3 \sqrt{2}}{(n \alpha^2 \gamma)^{1/2}}
\end{equation*}
must hold. We conclude similarly the upper bounds for the $\mathbb{L}_2$ test.
\end{proof}

\begin{proof}[Proof of Theorem~\ref{Thm:IntUpperBound}]
Recalling that $\tau =(n \alpha^2)^{-1/2}$, we first consider the expectation of our test statistic $D_n$. Writing $D_\tau(p):= \sum_{j=1}^d |p(j)-p_0(j)| \min( \tau, |p(j)-p_0(j)|)$, observe that
\begin{align}
\label{Eq:IntSignal}
	&\mathbb{E} D_n= \sum_{j=1}^d \{p(j)-p_0(j)\} \mathbb{E} \bigl\{ [\hat{p}_j - p_0(j)]_{-\tau}^\tau \bigr\} \nonumber \\
	& \geq \sum_{j=1}^d |p(j)-p_0(j)| \min\{ \tau, |p(j)-p_0(j)| \} \mathbb{P} \bigl( \mathrm{sign}( \hat{p}_j - p(j) ) = \mathrm{sign}(p(j) - p_0(j)) \bigr) \nonumber \\
	& \geq\sum_{j=1}^d |p(j)-p_0(j)| \min\{ \tau, |p(j)-p_0(j)| \} \nonumber \\
	& \hspace{50pt} \times \mathbb{P} \biggl( \frac{2}{n\alpha} \sum_{i=1}^n W_{ij} \geq \frac{1}{(n \alpha^2)^{1/2}} \biggr) \mathbb{P} \biggl( \biggl| \frac{1}{n} \sum_{i=1}^n \{ \mathbbm{1}_{\{X_i=j\}} - p(j) \} \biggr| \leq \frac{1}{(n \alpha^2)^{1/2}} \biggr) \nonumber \\
	& \geq ( 1 - \alpha^2/4) \mathbb{P} \biggl( \frac{2}{n^{1/2}} \sum_{i=1}^n W_{i1} \geq 1 \biggr) D_\tau(p).
\end{align}
Since we consider $\alpha \in (0,1]$, it is now clear that there exists a universal constant $c \in (0,1)$ such that $\mathbb{E} D_n \geq c D_\tau(p)$. Moreover, under $H_0$ we have that $\mathbb{E} D_n = 0$.

We now turn to the variance of $D_n$. Since the function $x \mapsto [x]_{-\tau}^\tau$ is Lipschitz, we have that
\begin{align}
\label{Eq:SmallVar}
	\mathrm{Var} \bigl( [\hat{p}_j - p_0(j)]_{-\tau}^\tau \bigr) &\leq \mathbb{E} \Bigl\{ \Bigl([\hat{p}_j - p_0(j)]_{-\tau}^\tau - [p(j)- p_0(j)]_{-\tau}^\tau \Bigr)^2 \Bigr\} \nonumber \\
	& \leq \mathrm{Var} ( \hat{p}_j) \leq \frac{1}{n} + \frac{8}{n \alpha^2} \leq \frac{9}{n \alpha^2}.
\end{align}
On the other hand, when $|p(j)-p_0(j)|$ is large, we can prove a tighter bound. Indeed, using a concentration inequality, we establish that
\begin{equation}
\label{Eq:ConcIneq}
|\hat p(j) - p(j)| \leq v, \quad \text{with probability larger than } 1- 4 \exp \Bigl(- \frac {n \alpha^2 v^2}{32} \Bigr).
\end{equation}
Thus, when $p(j)-p_0(j) \geq 2 \tau$, we have
\begin{align}
\label{Eq:LargeVar}
	\mathrm{Var} & \bigl( [\hat{p}_j - p_0(j)]_{-\tau}^\tau \bigr) \leq \mathbb{E} \bigl\{ \bigl( \tau - [\hat{p}_j - p_0(j)]_{-\tau}^\tau \bigr)^2 \bigr\} \leq 4 \tau^2 \mathbb{P}( \hat{p}_j - p_0(j) \leq \tau) \nonumber \\
	& \leq 16 \tau^2 \exp \Bigl( - \frac{n \alpha^2}{32} \{p(j)-p_0(j)-\tau\}^2 \Bigr) \leq \frac{16}{n \alpha^2} \exp \Bigl( - \frac{n \alpha^2 \{p(j)-p_0(j)\}^2}{128} \Bigr),
\end{align}
and we can similarly prove the same bound when $p(j) - p_0(j) \leq -2 \tau$. Using~\eqref{Eq:SmallVar} and~\eqref{Eq:LargeVar}, we can see that, for any value of $p(j)-p_0(j)$, we have
\begin{equation}
\label{Eq:IntDiagVar}
	\mathrm{Var} \bigl( [\hat{p}_j - p_0(j)]_{-\tau}^\tau \bigr) \leq \frac{16}{n \alpha^2} \exp\Bigl( - \frac{n \alpha^2 \{p(j)-p_0(j)\}^2}{128} \Bigr).
\end{equation}
For $j \in [d]$ we will write $P_j:= [ \hat{p}_{j} - p_0(j)]_{-\tau}^\tau$ and, for $i \in [n+1]$ and $j' \in [d]$ we will write $\mathbb{E}_i ( \cdot ):= \mathbb{E}( \cdot | X_1,\ldots,X_{i-1} )$ and $\mathbb{E}_i^j ( \cdot ):= \mathbb{E} (\cdot \times  \mathbbm{1}_{\{X_i = j\}} | X_1,\ldots,X_{i-1})/p(j)$ for conditional expectations. We will use the fact that $\mathbb{E}_i^{j_1}(P_j) = \mathbb{E}_i^{j_2}(P_j)$ almost surely for any $j_1,j_2 \neq j$ and $i \in [n+1]$. For $j_1,j_2 \in [d]$ such that $j_1 \neq j_2$, we now consider
\begin{align}
\label{Eq:CovIdentity}
	&\mathrm{Cov}\bigl( P_{j_1}, P_{j_2} \bigr) = \mathrm{Cov} \Bigl( \mathbb{E}_{n+1} \bigl( P_{j_1} \bigr) , \mathbb{E}_{n+1} \bigl( P_{j_2} \bigr) \Bigr) \nonumber \\
	&= \sum_{i=1}^n \mathbb{E} \bigl\{ \mathbb{E}_{i+1} \bigl( P_{j_1} \bigr) \mathbb{E}_{i+1} \bigl( P_{j_2} \bigr) - \mathbb{E}_i \bigl( P_{j_1} \bigr) \mathbb{E}_i \bigl( P_{j_2} \bigr)  \bigr\}  \nonumber \\
	& = \sum_{i=1}^n \mathbb{E} \Bigl[ p(j_1) \mathbb{E}_{i}^{j_1} (P_{j_1} ) \mathbb{E}_{i}^{j_1} (P_{j_2} ) + p(j_2) \mathbb{E}_{i}^{j_2} (P_{j_1} ) \mathbb{E}_{i}^{j_2} (P_{j_2} ) + \{1-p(j_1)-p(j_2)\} \mathbb{E}_{i}^{j_2} (P_{j_1} ) \mathbb{E}_{i}^{j_1} (P_{j_2} )  \nonumber \\
	& \hspace{50pt} - \bigl\{ p(j_1) \mathbb{E}_{i}^{j_1}(P_{j_1}) + (1-p(j_1)) \mathbb{E}_{i}^{j_2}(P_{j_1}) \bigr\} \bigl\{ p(j_2) \mathbb{E}_{i}^{j_2}(P_{j_2}) + (1-p(j_2)) \mathbb{E}_{i}^{j_1}(P_{j_2}) \bigr\}  \Bigr]  \nonumber \\
	& = - \sum_{i=1}^n p(j_1)p(j_2) \mathbb{E} \Bigl[ \bigl\{ \mathbb{E}_{i}^{j_1} (P_{j_1} ) - \mathbb{E}_{i}^{j_2} (P_{j_1} ) \bigr\} \bigl\{ \mathbb{E}_{i}^{j_2} (P_{j_2} ) - \mathbb{E}_{i}^{j_1} (P_{j_2} ) \bigr\} \Bigr]  \nonumber \\
	& = -n p(j_1)p(j_2) \mathbb{E} \bigl[ \bigl\{ [n^{-1}+ \hat{p}_{j_1} - p_0(j_1) ]_{-\tau}^\tau - [\hat{p}_{j_1} - p_0(j_1) ]_{-\tau}^\tau \bigr\}  \nonumber \\
	& \hspace{110pt} \times \bigl\{ [ \hat{p}_{j_2} - p_0(j_2) ]_{-\tau}^\tau - [\hat{p}_{j_2} - p_0(j_2) - n^{-1} ]_{-\tau}^\tau \bigr\} \bigm| X_1=j_2 \bigr].
\end{align}
We can therefore always say that, when $j_1 \neq j_2$, we have
\begin{equation}
\label{Eq:SmallCov}
	|\mathrm{Cov}([\hat{p}_{j_1} - p_0(j_1)]_{-\tau}^\tau , [\hat{p}_{j_2} - p_0(j_2)]_{-\tau}^\tau )| \leq p(j_1)p(j_2)/n.
\end{equation}
However, as before, tighter bound are available when $\max(|p(j_1)-p_0(j_1)|,|p(j_2)-p_0(j_2)|)$ is large. Indeed, if $j\in[d]$ is such that $|p(j)-p_0(j)| \geq 2(\tau+1/n)$, then, by~\eqref{Eq:ConcIneq} we have
\begin{align}
\label{Eq:LargeCov}
	\mathbb{E} &\bigl[ \bigl\{ [ \hat{p}_{j} - p_0(j) ]_{-\tau}^\tau - [\hat{p}_{j} - p_0(j) - n^{-1} ]_{-\tau}^\tau \bigr\}^2 \bigm| X_1=j \bigr] \nonumber \\
	& \leq \frac{1}{n^2} \mathbb{P} \biggl( \frac{1}{n} \sum_{i=2}^n \mathbbm{1}_{\{X_1=j\}} + \frac{2}{n\alpha} \sum_{i=1}^n W_{ij} - p_0(j) \leq \tau \biggr) \nonumber \\
	& \leq \frac{1}{n^2} \mathbb{P} \biggl( \biggl| \frac{1}{n} \sum_{i=2}^n \bigl\{ \mathbbm{1}_{\{X_1=j\}} - p(j) \bigr\} + \frac{2}{n\alpha} \sum_{i=1}^n W_{ij} \biggr|  \geq p(j) - p_0(j) - \tau - \frac{1}{n} \biggr) \nonumber \\
	& \leq \frac{4}{n^2} \exp\biggl( - \frac{n \alpha^2}{32} \{ p(j) - p_0(j) -\tau - 1/n\}^2 \biggr) \leq \frac{4}{n^2} \exp \biggl( - \frac{n \alpha^2 \{p(j)-p_0(j)\}^2}{128} \biggr).
\end{align}
It now follows from Cauchy--Schwarz,~\eqref{Eq:CovIdentity},~\eqref{Eq:SmallCov} and~\eqref{Eq:LargeCov} that, whenever $j_1 \neq j_2$, we have
\begin{align}
\label{Eq:IntCov}
	|\mathrm{Cov}([\hat{p}_{j_1} - p_0&(j_1)]_{-\tau}^\tau , [\hat{p}_{j_2} - p_0(j_2)]_{-\tau}^\tau )| \nonumber \\
	& \leq \frac{4}{n} p(j_1)p(j_2) \exp \biggl( - \frac{n \alpha^2}{256} \bigl[ \{p(j_1)-p_0(j_1)\}^2 + \{p(j_2)-p_0(j_2)\}^2 \bigr] \biggr).
\end{align}
It now follows from~\eqref{Eq:IntDiagVar},~\eqref{Eq:IntCov} and the fact that $\sup_{x \geq 0} \frac{x e^{-x^2/128}}{ x \wedge 1} = 8 e^{-1/2}$, that
\begin{align}
\label{Eq:IntVar}
	\mathrm{Var}(D_n) &= \mathbb{E} \Bigl\{ \mathrm{Var} \bigl( D_n | Z_1,\ldots,Z_n \bigr) \Bigr\} + \mathrm{Var} \Bigl( \mathbb{E} \bigl\{ D_n | Z_1, \ldots, Z_n \bigr\} \Bigr) \nonumber \\
	& = \frac{c_\alpha^2 \tau^2}{n} + \mathrm{Var} \biggl( \sum_{j=1}^d \{p(j) - p_0(j) \} [ \hat{p}_j - p_0(j) ]_{-\tau}^\tau \biggr) \nonumber \\
	& \leq \frac{c_\alpha^2 \tau^2}{n} + \frac{16}{n \alpha^2} \sum_{j=1}^d \{p(j)-p_0(j)\}^2  \exp\Bigl( - \frac{n \alpha^2 \{p(j)-p_0(j)\}^2}{128} \Bigr) \nonumber \\
	& \hspace{50pt} + \frac{4}{n}  \biggl\{ \sum_{j=1}^d |p(j)-p_0(j)| p (j) \exp \biggl( - \frac{n \alpha^2 \{p(j)-p_0(j)\}^2}{256} \biggr) \biggr\}^2 \nonumber \\
	& \leq \frac{c_\alpha^2 \tau^2}{n} + \frac{20}{n \alpha^2} \sum_{j=1}^d \{p(j)-p_0(j)\}^2  \exp\Bigl( - \frac{n \alpha^2 \{p(j)-p_0(j)\}^2}{128} \Bigr) \nonumber \\
	& \leq \frac{c_\alpha^2 \tau^2}{n} + \frac{160}{n \alpha^2 e^{1/2}}  D_\tau(p) \leq \frac{(e+1)^2}{(e-1)^2(n \alpha^2)^2}  + \frac{160D_\tau(p)}{e^{1/2} n \alpha^2}.
\end{align}

Under $H_0$, we can now see that
\[
	\mathbb{P}(D_B \geq C_3) = \mathbb{P}\Bigl( D_n \geq \frac{e+1}{e-1}  \frac{(4/\gamma)^{1/2}}{n \alpha^2} \Bigl) \leq \frac{\gamma}{4}.
\]
As we have already shown in the proof of Theorem~\ref{Thm:NonIntUpperBound}, we also have that $\mathbb{P}(T_B \geq C_{2,B}) \leq \gamma/4$ under $H_0$, so that the Type I error of our combined test $\psi_B$ is bounded above by $\gamma/2$. Now, suppose that $p$ is such that
\[
	D_\tau(p) \geq \max\biggl\{ \frac{(4/\gamma)^{1/2}}{c}  \frac{2(e+1)}{e-1}, \frac{2560}{e^{1/2}c^2 \gamma} \biggr\} \frac{1}{n \alpha^2},
\]
where $c$ is the universal constant defined just after~\eqref{Eq:IntSignal}. For such $p$, it follows from~\eqref{Eq:IntSignal} and~\eqref{Eq:IntVar} that
\[
	\mathbb{P}(D_n < C_3) \leq \frac{\mathrm{Var} D_n}{ \{c D_\tau(p) - C_3\}^2} \leq \frac{\gamma}{2}.
\]
Now, under $H_1(\delta,\mathbb{L}_2)$, we have
\begin{align*}
	D_\tau(p) &= \sum_{j=1}^d \{p(j)-p_0(j)\}^2 \min(1,\tau/|p(j)-p_0(j)|) \\
	&\geq \min( \|p-p_0\|_2^2, \tau \|p-p_0\|_2) \geq \min(\delta^2, \tau \delta)
\end{align*}
This proves that
\[
	\mathcal{E}_{n,\alpha}(p_0, \mathbb{L}_2) \leq \max\biggl\{ \frac{(4/\gamma)^{1/2}}{c}  \frac{2(e+1)}{e-1}, \frac{2560}{e^{1/2}c^2 \gamma} \biggr\} \frac{1}{(n \alpha^2)^{1/2}}.
\]

We now prove the $\mathbb{L}_1$ result. Let $\emptyset \neq B \subseteq [d]$ be given, and suppose that
\[
	\delta \geq 8 \max\biggl[ \biggl( \frac{|B|}{n \alpha^2} \biggr)^{1/2}\max\biggl\{ \frac{(4/\gamma)^{1/2}}{c}  \frac{2(e+1)}{e-1}, \frac{2560}{e^{1/2}c^2 \gamma} \biggr\} , p_0(B^c) \biggr].
\]
Then, under $H_1(\delta,\mathbb{L}_1)$, at least one of
\[
	\sum_{j \in B} |p(j) - p_0(j)| \geq \biggl( \frac{|B|}{n \alpha^2} \biggr)^{1/2} \max\biggl\{ \frac{(4/\gamma)^{1/2}}{c}  \frac{2(e+1)}{e-1}, \frac{2560}{e^{1/2}c^2 \gamma} \biggr\}
\]
or
\[
	\sum_{j \in B^c} |p(j) - p_0(j)| \geq 2p_0(B^c) + \frac{6 + 3 \sqrt{2}}{(n \alpha^2 \gamma)^{1/2}}
\]
holds. If the second of these holds, then, as in the proof of Theorem~\ref{Thm:NonIntUpperBound}, we have $\mathbb{P}(T_B < C_{2,B}) \leq \gamma/2$. On the other hand, if the first holds, then we have
\begin{align*}
	\|p-p_0\|_2^2 &\geq \sum_{j \in B} \{p(j)-p_0(j)\}^2 \geq \frac{1}{|B|}  \biggl( \sum_{j \in B} |p(j) - p_0(j)| \biggr)^2 \\
	&\geq \max\biggl\{ \frac{(4/\gamma)^{1/2}}{c}  \frac{2(e+1)}{e-1}, \frac{2560}{e^{1/2}c^2 \gamma} \biggr\}^2 \frac{1}{n \alpha^2},
\end{align*}
and our interactive test rejects $H_0$ with probability at least $\gamma/2$. Thus,
\[
	\mathcal{E}_{n,\alpha}^\mathrm{I}(p_0, \mathbb{L}_1) \leq 8 \max\biggl[ \biggl( \frac{|B|}{n \alpha^2} \biggr)^{1/2}\max\biggl\{ \frac{(4/\gamma)^{1/2}}{c}  \frac{2(e+1)}{e-1}, \frac{2560}{e^{1/2}c^2 \gamma} \biggr\} , p_0(B^c) \biggr].
\]

\end{proof}

\begin{proof}[Proof of Proposition~\ref{Prop:LowerBound}] The minimax risk for testing is
\begin{align*}
\mathcal{R}_{n,\alpha}(p_0,\delta) & \geq \inf_{Q \in \mathcal{Q}_\alpha} \inf_{\phi \in \Phi_Q} \sup_{p_\xi \in H_1( \delta ), \xi \in \mathcal{V}} \bigl\{ \mathbb{E}_{p_0} (\phi) + \mathbb{E}_p(1-\phi) \bigr\}\\
& \geq \inf_{Q \in \mathcal{Q}_\alpha} \inf_{\phi \in \Phi_Q} \bigl\{ \mathbb{E}_{p_0} (\phi) + E_\xi [\mathbb{E}_{p_\xi}(1-\phi) \cdot I_{p_\xi \in H_1(\delta)} ]\bigr\},
\end{align*}
where $E_\xi$ is the average with respect to $\xi$ uniformly distributed over $ \mathcal{V}$.

Denote by $QP_0^n$ and $QP_\xi^n$ the likelihood of the sample $Z_1,...,Z_n$ when the original sample is distributed according to $p_0$ and $p_\xi$, respectively. We write
\begin{align*}
E_\xi [\mathbb{E}_{p_\xi}(1-\phi)& \cdot I_{p_\xi \in H_1(\delta)} ]  = E_\xi \left[\mathbb{E}_{p_0} \frac {QP_\xi^n}{QP_0^n} (1 -  I_{p_\xi \not \in H_1(\delta)} )  \cdot (1-\phi)\right] \\
& =  \mathbb{E}_{p_0} \left[  E_\xi \frac {QP_\xi^n}{QP_0^n} (1 -  I_{p_\xi \not \in H_1(\delta)} ) \cdot (1-\phi)  \right] \geq  \mathbb{E}_{p_0} \left[  E_\xi \frac {QP_\xi^n}{QP_0^n} (1-\phi)  \right] - \gamma_1.
\end{align*}
Back to the minimax risk
\begin{align*}
\mathcal{R}_{n,\alpha}(p_0,\delta) & \geq \inf_{Q \in \mathcal{Q}_\alpha} \inf_{\phi \in \Phi_Q} \mathbb{E}_{p_0} (\phi) +  \mathbb{E}_{p_0} \left[  E_\xi \frac {QP_\xi^n}{QP_0^n} (1-\phi)  \right] - \gamma_1\\
& \geq \inf_{Q \in \mathcal{Q}_\alpha} (1- \eta) \mathbb{P}_{p_0} \left(E_\xi \frac {QP_\xi^n}{QP_0^n} \geq 1-\eta \right) - \gamma_1\\
& \geq \inf_{Q \in \mathcal{Q}_\alpha} (1- \eta) \left( 1 - \frac 1\eta TV(QP_0^n, E_\xi QP_\xi^n ) \right) - \gamma_1,
\end{align*}
for arbitrary $\eta$ in (0,1).
\end{proof}


\begin{proof}[Proof of Theorem~\ref{Thm:InfoBounds}]
For general sequentially interactive mechanisms, we use the convexity of the Kullback--Leibler discrepancy and the fact that the Kullback--Leibler discrepancy is bounded above by the $\chi^2$ discrepancy to get
\begin{align*}
& KL(QP_0^n, E_\xi Q P_\xi^n)  \leq E_\xi \int m^0(z) \log \frac{m^\xi(z)}{m^0(z)} dz\\
& = \sum_{i=1}^n E_\xi \E_{p_0}\left[\int  \log \frac{m_i^\xi(z_i|Z_1,...,Z_{i-1})}{m_i^0(z_i|Z_1,...,Z_{i-1})} m_i^0(z_i|Z_1,...,Z_{i-1}) dz_i \right]\\
& \leq \sum_{i=1}^n E_\xi \E_{p_0}\left[\int  \frac{(m_i^\xi-m_i^0)^2(z_i|Z_1,...,Z_{i-1})}{m_i^0(z_i|Z_1,...,Z_{i-1})}  dz_i \right]\\
& = \sum_{i=1}^n E_\xi \E_{p_0}\left[  (p_\xi-p_0)^\top \int \frac{q_i(z_i|\cdot, Z_1,...,Z_{i-1}) q_i(z_i|\cdot, Z_1,...,Z_{i-1})^\top }{{m_i^0(z_i|Z_1,...,Z_{i-1})}} dz_i (p_\xi - p_0) \right]\\
& = E_\xi \left[ (p_\xi-p_0)^\top \Omega (p_\xi-p_0)\right].
\end{align*}
In the particular case of noninteractive mechanisms, we have
\begin{align*}
& \chi^2 (QP_0^n, E_\xi Q P_\xi^n)  = \E_{p_0} \left[ \left( E_\xi \frac{m_1^\xi(Z_1) \cdot ... \cdot m_n^\xi (Z_n)}{m_1^0(Z_1) \cdot ... \cdot m_n^0 (Z_n)} \right)^2 \right] - 1 \\
& = \E_{p_0} \left[ E_{\xi,\xi'} \left( \frac{m_1^\xi(Z_1) \cdot ... \cdot m_n^\xi (Z_n)}{m_1^0(Z_1) \cdot ... \cdot m_n^0 (Z_n)}  \frac{m_1^{\xi'}(Z_1) \cdot ... \cdot m_n^{\xi'} (Z_n)}{m_1^0(Z_1) \cdot ... \cdot m_n^0 (Z_n)}\right) \right] - 1\\
& =  E_{\xi,\xi'} \prod_{i=1}^n \E_{p_0} \left[ \left( 1 + \frac{m_i^\xi(Z_i)-m_i^0(Z_i)}{m_i^0(Z_i)} \right)\left( 1 + \frac{m_i^\xi(Z_i)-m_i^0(Z_i)}{m_i^0(Z_i)} \right) \right]-1\\
& =  E_{\xi,\xi'} \prod_{i=1}^n \left( 1 + \E_{p_0} \left[ \frac{m_i^\xi(Z_i)-m_i^0(Z_i)}{m_i^0(Z_i)}  \frac{m_i^\xi(Z_i)-m_i^0(Z_i)}{m_i^0(Z_i)} \right] \right)-1.
\end{align*}
Indeed, $\E_{p_0} [(m_i^\xi(Z_i)-m_i^0(Z_i))/m_i^0(Z_i)] = 0 $. Moreover,
\begin{align*}
& \chi^2 (QP_0^n, E_\xi Q P_\xi^n)  \leq E_{\xi,\xi'} \exp \left( \sum_{i=1}^n
\E_{p_0} (\frac{m_i^\xi(Z_i)}{m_i^0(Z_i)}- 1)(\frac{m_i^\xi(Z_i)}{m_i^0(Z_i)} -1)\right) - 1 \\
& \leq E_{\xi,\xi'}  \exp \left( (p_\xi-p_0)^\top \sum_{i=1}^n
\E_{p_0} \left[ (\frac{q_i^\xi(Z_i|\cdot)}{m_i^0(Z_i)}- 1)(\frac{q_i^{\xi'}(Z_i|\cdot)^\top}{m_i^0(Z_i)} -1) \right]
 (p_{\xi'} - p_0)\right) - 1\\
 & \leq E_{\xi,\xi'}\left[ \exp\left( (p_\xi-p_0)^\top \Omega (p_{\xi'} - p_0) \right)\right] - 1.
\end{align*}
\end{proof}

\begin{proof}[Proof of Theorem~\ref{Thm:NonIntLowerBound}]
For $i \in [n]$, write $q_i(j|\cdot)$ for the density of $Z_i | \{X_i=j\}$, and write
\[
	m_0^i(z):= \sum_{j=1}^d q_i(z|j) p_0(j).
\]
For $j_* \in [d]$ let $B=\{2,\ldots,j_*+1\}$, and for $j,j' \in B$ and $i \in [n]$ write
\[
	\omega_{j j'}^i = \int m_0^i(z) \Bigl\{ \frac{q_i(z|j)}{m_0^i(z)} -1 \Bigr\} \Bigl\{ \frac{q_i(z|j')}{m_0^i(z)} -1 \Bigr\} \,dz.
\]
For each $i \in [n]$, the matrix $\Omega_i:=(\omega_{jj'}^i)_{j,j' \in B}$ is a covariance matrix so it is symmetric and non-negative definite. Writing $\bar{\Omega}:= n^{-1} \sum_{i=1}^n \Omega_i$, then $\bar{\Omega}$ is also symmetric and non-negative definite and hence has real eigenvalues $0 \leq \lambda_1 \leq \ldots \leq \lambda_{j_*}$ and associated eigenvectors $v_1,\ldots,v_{j_*}$. Since $Q$ is $\alpha$-LDP we have that
\[
	\mathrm{trace}( \bar{\Omega}) = \frac{1}{n} \sum_{i=1}^n \mathrm{trace}(\Omega_i) = \frac{1}{n} \sum_{i=1}^n \sum_{j \in B} \int m_0^i(z) \Bigl\{ \frac{q_i(z|j)}{m_0^i(z)} -1 \Bigr\}^2 \,dz \leq (e^\alpha -1)^2 j_*.
\]
Now if we take $j_0:= \max\{ j \in B : \lambda_j \leq 2 (e^\alpha-1)^2 \}$ we have that $j_0 > j_*/2-1$.
Indeed, if we had $j_0 \leq j_*/2 - 1$ then
$$
\sum_{j>j_0}^{j_*} \lambda_j > (j_* - j_0) \cdot 2(e^\alpha - 1)^2 \geq (j_* + 2) (e^\alpha - 1)^2,
$$
which is in contradiction with the fact that $\sum_j \lambda_j \leq j_* (e^\alpha - 1)^2$.

Given a sequence $\xi=(\xi_1,\ldots,\xi_{j_0}) \in \{-1,1\}^{j_0}$ define $\delta_\xi^j :=\sum_{k=1}^{j_0} \xi_k v_{kj}$ for $j \in B$, define $\delta_\xi^+ :=\sum_{j \in B} \delta_j$ and, given $\epsilon>0$, define
\[
	p_\xi(j) := \left\{ \begin{array}{cc} p_0(j)(1-\epsilon \delta_\xi^+) + \epsilon \delta_\xi^j  & \text{if } j \in B \\ p_0(j)(1- \epsilon \delta_\xi^+) & \text{otherwise} \end{array} \right..
\]
Note that we have $\sum_{j=1}^d p_\xi(j)  =1$. Write $\Xi_\epsilon \subset \{-1,1\}^{j_0}$ for the set of all sequences $\xi$ such that $|\delta_\xi^+| \leq 1/(2\epsilon)$ and $\max_{j \in B} |\delta_\xi^j| \leq p_0(j_*+1)/(2\epsilon)$. Then, for $\xi \in \Xi_\epsilon$, we have $p_\xi \in \mathcal{P}_d$.
Given $\xi \in \Xi_\epsilon$ write
\begin{align*}
	m_\xi^i(z) &= \sum_{j=1}^d q_i(z|j) p_\xi(j) = (1-\epsilon \delta_\xi^+) m_0^i(z) + \epsilon \sum_{j \in B} \delta_\xi^j q_i(z|j) \\
	&= m_0^i(z) \biggl[ 1 + \epsilon \sum_{j \in B} \delta_\xi^j \biggl\{ \frac{q_i(z|j)}{m_0^i(z)} -1 \biggr\} \biggr] = m_0^i(z) \biggl[ 1 + \epsilon \delta_\xi^T \biggl\{ \frac{q_i(z | \cdot)}{m_0^i(z)} - \mathbf{1} \biggr\} \biggr]
\end{align*}
where we write $\mathbf{1}=(1,\ldots,1) \in \mathbb{R}^{j_*}$ for the constant vector, $q_i(z|\cdot) = (q_i(z|2),\ldots,q_i(z|j_*+1))$ and $\delta_\xi=(\delta_\xi^2,\ldots,\delta_\xi^{j_*}) = \sum_{k=1}^{j_0} \xi_k v_k $. Let $\eta$ be a uniformly random element of $\Xi_\epsilon$, and define
\[
	Y = E_\eta \biggl[ \frac{m_\eta^1(Z_1) \ldots m_\eta^n(Z_n)}{m_0^1(Z_1)\ldots m_0^n(Z_n)} \biggr] -1.
\]
Let $\eta'$ be an independent copy of $\eta$, and let $\xi,\xi'$ be two independent sequences of Rademacher random variables. Then, using the facts that $1+x \leq e^x$ for all $x \in \mathbb{R}$ and $\Xi_\epsilon = - \Xi_\epsilon$, we have
\begin{align*}
	&\mathbb{E}_{p_0}(Y^2) = E_{\eta,\eta'} \bigg[\int \frac{m_\eta^1(z_1) m_{\eta'}^1(z_1) \ldots m_\eta^n(z_n) m_{\eta'}^n(z_n) }{m_0^1(z_1)  \ldots m_0^n(z_n) } \,dz_1 \ldots \, d z_n \biggr] -1 \\
	&= E_{\eta,\eta'} \bigl\{ \bigl( 1 + \epsilon^2 \delta_\eta^T \Omega_1 \delta_{\eta'} \bigr) \ldots \bigl( 1 + \epsilon^2 \delta_\eta^T \Omega_n \delta_{\eta'} \bigr)  \bigr\} -1 \leq  E_{\eta,\eta'} \Bigl\{ \exp \bigl( n \epsilon^2 \delta_\eta^T \bar{\Omega} \delta_{\eta'} \bigr) \Bigr\} -1 \\
	& = E_{\eta,\eta'} \biggl\{ \exp \biggl(n \epsilon^2 \sum_{k=1}^{j_0} \eta_k \eta_{k}' \lambda_k \biggr) - 1 \biggr\} = E_{\eta,\eta'} \biggl\{ \sum_{\ell=1}^\infty \frac{1}{(2\ell)!} \biggl( n \epsilon^2 \sum_{k=1}^{j_0} \eta_k \eta_{k}' \lambda_k \biggr)^{2\ell} \biggr\} \\
	& \leq \frac{1}{P_\xi(\xi \in \Xi_\epsilon)^2} E_{\xi,\xi'} \biggl\{ \sum_{\ell=1}^\infty \frac{1}{(2\ell)!} \biggl( n \epsilon^2 \sum_{k=1}^{j_0} \xi_k \xi_{k}' \lambda_k \biggr)^{2\ell} \biggr\} \\
 & = \frac{1}{P_\xi(\xi \in \Xi_\epsilon)^2} E_{\xi,\xi'} \biggl\{ \exp \biggl( n \epsilon^2 \sum_{k=1}^{j_0} \xi_k \xi_k' \lambda_k \biggr) -1  \biggr\} \\
	& \leq \frac{1}{P_\xi(\xi \in \Xi_\epsilon)^2} \biggl\{ \exp \biggl( \frac{n^2 \epsilon^4}{2} \sum_{k=1}^{j_0} \lambda_k^2 \biggr) -1 \biggr\} \leq \frac{\exp \bigl( 2 n^2 \epsilon^4 (e^\alpha-1)^4 j_0 \bigr) -1 }{P_\xi(\xi \in \Xi_\epsilon)^2}.
\end{align*}

We now study $P_\xi(\xi \in \Xi_\epsilon)$. Note that for each $j \in B$ the random variable $\delta_\xi^j$ is subgaussian with variance proxy $\sum_{k=1}^{j_0} v_{kj}^2 \leq 1$. We therefore have \citep[][Theorem~11.8]{BGM2013}
\begin{align*}
	E_\xi \Bigl\{ \max_{j \in B} |\delta_\xi^j| \Bigr\} \leq \{2 \log (2j_*) \}^{1/2} \quad \text{and} \quad \mathrm{Var}_\xi \Bigl( \max_{j \in B} |\delta_\xi^j| \Bigr) \leq 8 \{2 \log (2j_*) \}^{1/2} + 2.
\end{align*}
Hence, $P_\xi( \max_{j \in B} |\delta_\xi^j| \geq 2 \log^{1/2}(2 j_*) ) \rightarrow 0$ as $d \rightarrow \infty$. Now $\delta_\xi^+$ is subgaussian with variance proxy
\[
	\sum_{k=1}^{j_0} \biggl( \sum_{j \in B} v_{kj} \biggr)^2 = \sum_{k=1}^{j_0} (v_k^T \mathbf{1})^2 \leq \| \mathbf{1} \|^2 \leq j_*.
\]
We may therefore take
\[
	\epsilon \asymp \min \biggl\{ \frac{1}{j_*^{1/4} (n \alpha^2)^{1/2}},  \frac{p_0(j_*+1)}{\log^{1/2}(j_*)}, \frac{1}{j_*^{1/2}} \biggr\}.
\]
Now
\begin{align*}
	\|p_\xi - p_0 \|_1 &= \epsilon \sum_{j \in B} | \delta_\xi^j - p_0(j) \delta_\xi^+| + \epsilon \sum_{j \in B^c} p_0(j) |\delta_\xi^+| \geq \epsilon \sum_{j \in B} |\delta_\xi^j| - \epsilon |\delta_\xi^+|.
\end{align*}
By the Khintchine inequality we have that
\begin{align*}
	\sum_{j \in B} E_\xi |\delta_\xi^j| &= \sum_{j \in B} E_\xi \biggl| \sum_{k=1}^{j_0} \xi_k v_{kj} \biggr| \geq \frac{1}{2^{1/2}} \sum_{j \in B} \biggl( \sum_{k=1}^{j_0} v_{kj}^2 \biggr)^{1/2} \geq \frac{1}{2^{3/2}} \sum_{j \in B} \mathbbm{1}_{ \{ \sum_{k=1}^{j_0} v_{kj}^2 \geq 1/4 \}} \\
	& \geq \frac{1}{2^{3/2}} \biggl( \sum_{j \in B} \sum_{k=1}^{j_0} v_{kj}^2 - \frac{j_*}{4} \biggr) = \frac{j_0 - j_*/4}{2^{3/2}} \geq \frac{j_*}{24\sqrt{2}},
\end{align*}
where the final inequality follows from the facts that $j_0 > j_*/2 -1$ and $j_0 \in \mathbb{N}$. Now
%
\begin{align*}
	\mathrm{Var}_\xi \biggl( \sum_{j \in B} |\delta_\xi^j| \biggr) = \mathrm{Var}_\xi \biggl( \sum_{j \in B} \biggl| \sum_{k=1}^{j_0} \xi_k v_{kj} \biggr| \biggr)\leq E_\xi \left[ \biggl( \sum_{j \in B} \biggl| \sum_{k=1}^{j_0} \xi_k v_{kj} \biggr| \biggr)^2 \right].
\end{align*}
Denote by $V = \sum_{j \in B} \biggl| \sum_{k=1}^{j_0} \xi_k v_{kj} \biggr|  $.
We can prove that, for $t>1$,
\begin{eqnarray*}
P_\xi\left( V \geq t \sqrt{2 \log(j^*)} \right)
&\leq & \sum_{j \in B} P_\xi \left( \biggl| \sum_{k=1}^{j_0} \xi_k v_{kj} \biggr| \geq t \sqrt{2 \log(j^*)} \right)\\
&\leq & j^* \exp\left( - t^2 \log(j^*) \right) \leq \exp(-(t^2-1) \log(j^*)).
\end{eqnarray*}
Now,
$
E_\xi[V^2] = \int_0^\infty 2v P_\xi(V\geq v) dv \leq 2 j^* + 2 \int_{j^*}^\infty v \exp(-v^2/2 + j^*)dv
\lesssim j^*.
$

Moreover, $E_\xi|\delta_\xi^+| \leq j_*^{1/2}$. Writing $Z:= \frac{\sum_{j \in B} | \delta_\xi^j|}{ \sum_{j \in B} \mathbb{E} |\delta_\xi^j|}$ we have $\mathrm{Var}_\xi(Z) \leq 1152$ and hence that
\begin{align*}
	1 &= E_\xi Z  \leq \frac{1}{4} + 4612 P_\xi (1/4 \leq Z < 4612) + \frac{\mathbb{E}(Z^2)}{4612}  \leq \frac{1}{2} + 4612 P_\xi(Z \geq 1/4).
\end{align*}
Thus,
\begin{align*}
	P_\xi \biggl( \|p_\xi - p_0\|_1 \geq \frac{\epsilon j_*}{192 \sqrt{2}} \biggr) &\geq P_\xi \biggl( \sum_{j \in B} | \delta_\xi^j| \geq \frac{j_*}{96 \sqrt{2}} \biggr) - P_\xi \biggl( | \delta_\xi^+| > \frac{j_*}{192 \sqrt{2}} \biggr) \\
	& \geq \frac{1}{9224} - \frac{192 \sqrt{2}}{j_*^{1/2}} \geq \frac{1}{10000}
\end{align*}
for $j_*$ sufficiently large. Thus
\begin{align*}
	\mathcal{E}_{n,\alpha}^\mathrm{NI}(p_0, \mathbb{L}_1) &\gtrsim \epsilon j_* \gtrsim \min \biggl\{ \frac{j_*^{3/4}}{(n \alpha^2)^{1/2}}, \frac{j_*p_0(j_*+1)}{\log^{1/2}(2 j_*)}, j_*^{1/2} \biggr\} \\
& = \min \biggl\{ \frac{j_*^{3/4}}{(n \alpha^2)^{1/2}}, \frac{j_*p_0(j_*+1)}{\log^{1/2}(2 j_*)} \biggr\},
\end{align*}
and the result follows.

\bigskip

The proof for the $\mathbb{L}_2$ test follows the same lines. It is sufficient to bound from below $\|p_\xi-p_0\|_2$ with high probability. We have
\begin{align*}
P_\xi \left( \|p_\xi - p_0\|_2^2\geq \frac 1{144} \varepsilon^2 j_*\right) & \geq P_\xi \left(\varepsilon^2 \left\{ \sum_{j \in B} (\delta_\xi^j)^2 - 2 \delta_\xi^{+} \cdot \sum_{j \in B} \delta_\xi^j p_0(j) \right\} \geq \frac 1{144} \varepsilon^2 j_* \right) \\
& \geq P_\xi\left(\sum_{j \in B} (\delta_\xi^j)^2 \geq \frac 1{16} j_* \right) -P_\xi \left( 2 \delta_\xi^{+} \cdot \sum_{j \in B} \delta_\xi^j p_0(j) \geq \frac 1{18} j_*\right),
\end{align*}
for $j_*$ large enough. Moreover,
$\sum_{j \in B} E_\xi(\delta_\xi^j)^2 = \sum_{j \in B} \sum_k v^2_{kj} = j_0$ by orthonormality of the eigenvectors $v_j$ and
\[
E_\xi \left[\left( \sum_{j \in B} (\delta_\xi^j)^2 \right)^2 \right] = \left(\sum_{j \in B}\sum_{k =1}^{j_0} v^2_{kj} \right)^2 = j_0^2.
\]
Therefore, $P_\xi(\sum_{j \in B} (\delta_\xi^j)^2 \geq 2 j_0 ) \leq 1/4$. Denote by $Z = \sum_{j \in B} (\delta_\xi^j)^2$ We get
\[
1 = E_\xi (Z/\E Z) \leq \frac 14 + 2 P_\xi(Z \geq \E Z/4) + P_\xi(Z \geq 2 \E Z )\leq \frac 12 + 2 \cdot P_\xi(Z \geq j_0/4)
\]
meaning that $P_\xi(Z \geq j_*/16) \geq P_\xi(Z \geq j_0/4) \geq 1/4$ (as $j_0 \geq j_*/2 - 1 \geq j_*/4$ for $j_*$ large enough).
Also
\begin{align*}
P_\xi \left( 2 \delta_\xi^{+} \cdot \sum_{j \in B} \delta_\xi^j p_0(j) \geq \frac 1{18} j_*\right)
& \leq \frac {36}{ j_*} E_\xi \left[|\delta_\xi^+ \cdot \sum_{j \in B} \delta_\xi^j p_0(j)| \right] \\
& \leq \frac {36}{ j_*} \left(  E_\xi (\delta_\xi^+)^2 \cdot E_\xi (\sum_{j \in B} \delta_\xi^j p_0(j))^2 \right)^{1/2}\\
& \leq \frac {36}{ j_*} j_*^{1/2} \left( \sum_k \sum_j v^2_{kj} p_0(j)\right)^{1/2} \leq \frac {36}{j_*^{1/2}},
\end{align*}
which is less or equal to 1/5 for $j_*$ large enough. Thus
\[
	\mathcal{E}_{n,\alpha}^\mathrm{NI}(p_0, \mathbb{L}_2) \gtrsim \epsilon \sqrt{j_*} \gtrsim \min \biggl\{ \frac{j_*^{1/4}}{(n \alpha^2)^{1/2}}, \frac{j_*^{1/2} p_0(j_*+1)}{\log^{1/2}(2 j_*)}, 1 \biggr\} .
\]
\end{proof}

\begin{proof}[Proof of Theorem~\ref{Thm:IntLowerBound}]
Let us first prove the bounds for the $\mathbb{L}_2$ norm. When $\epsilon \in [0,1-1/d]$ we can define the probability vector
\[
	p = (1-\epsilon) p_0 + (0,\ldots,0,\epsilon),
\]
which satisfies $\|p-p_0\|_1 = \epsilon \{1-p_0(d)\} \leq \epsilon$ and
\[
	\|p-p_0\|_2 = \epsilon \biggl[ \{1-p_0(d)\}^2 + \sum_{j=1}^{d-1} p_0(j)^2 \biggr]^{1/2} \geq \epsilon ( 1- 1/d).
\]
Thus, using Theorem~1 of \citet{DJW2018} and taking $\epsilon \leq \frac{1}{\sqrt{8 n \alpha^2}}$, we have that
\[
	\| M_1 - M_0 \|_\mathrm{TV} \leq \frac{1}{\sqrt{2}}
\]
for any sequentially interactive privacy mechanism that takes $p_0$ to $M_0$ and $p$ to $M_1$. We can therefore establish a lower bound of the order of $(n \alpha^2)^{-1/2}$ for the $\mathbb{L}_2$ testing problem.

\bigskip

{\bf Proof of the lower bounds for the $\mathbb{L}_1$-risk, interactive mechanisms}
Fix $j_* \in [d]$ and write $B=\{1,\ldots,j_*\}$. Let $Q$ be a sequentially interactive, $\alpha$-LDP privacy mechanism, and for each $i \in [n], j \in [d]$ and $z_1,\ldots,z_{i-1},z$, write $q(z|j, z_1,\ldots,z_{i-1})$ for the conditional density of $Z_i$ given $X_i=j,Z_1=z_1,\ldots,Z_{i-1}=z_{i-1}$. For each $i \in [n]$ and $z_1,\ldots,z_{i-1}$ define the $j_* \times j_*$ matrix $\Omega_i(z_1,\ldots,z_{i-1})$ by
\begin{align*}
	&\Omega_i(z_1,\ldots,z_{i-1})_{j_1 j_2} \\
	& := \int \{p_0^T q_i(z | \cdot, z_1,\ldots,z_{i-1})\} \biggl( \frac{q_i(z | j_1, z_1,\ldots,z_{i-1} )}{p_0^T q_i(z | \cdot, z_1,\ldots,z_{i-1}) } -1 \biggr)\biggl( \frac{q_i(z | j_2, z_1,\ldots,z_{i-1} )}{p_0^T q_i(z | \cdot, z_1,\ldots,z_{i-1}) } -1 \biggr)^T  \,dz.
\end{align*}
Consider the $j_* \times j_*$ non-negative definite matrix
\[
	\Omega:= \mathbb{E}_{p_0} \biggl[ \sum_{i=1}^n \Omega_i(Z_1,\ldots,Z_{i-1}) \biggr],
\]
and write $\lambda_1 \geq \lambda_2 \geq \ldots \geq \lambda_{j_*} \geq 0$ for its eigenvalues and $v_1,\ldots,v_{j_*}$ for its associated eigenvectors, with $v_d=p_0$ and $\lambda_d=0$ if $j_*=d$. Given a sequence $\xi=(\xi_1,\ldots,\xi_{j_* \wedge(d-1)}) \in \{-1,1\}^{j_* \wedge (d-1)}$ define $\delta_\xi^j := \sum_{k=1}^{j_* \wedge (d-1)} \xi_k v_{kj}$ for $j \in B$ and define $\delta_\xi^+:= \sum_{j \in B} \delta_\xi^j$. Further, given $\epsilon >0$, set
\[
		p_\xi(j) := \left\{ \begin{array}{cc} (1- \epsilon \delta_\xi^+) p_0(j)+ \epsilon \delta_\xi^j & \text{if } j \in B \\ (1- \epsilon \delta_\xi^+) p_0(j) & \text{otherwise} \end{array} \right. .
\]
This sums to zero, and when $\epsilon \lesssim p_0(j_*)/\sqrt{\log (2j_*)}$ and $\xi$ is an i.i.d. Rademacher vector, then $p_\xi$ is also non-negative with high probability. Moreover, for each $i \in [n]$ and $z_1,\ldots,z_i$, we have
\[
	\biggl| \frac{(p_\xi - p_0)^T q_i(z_i | \cdot, z_1,\ldots,z_{i-1})}{p_0^T q_i(z_i | \cdot, z_1,\ldots,z_{i-1}) } \biggr| \leq e^{2\alpha} \|p_\xi - p_0\|_1 \leq 2 e^{2 \alpha} \epsilon \sum_{j \in B} \biggl| \sum_{k=1}^{j_* \wedge(d-1)} \xi_k v_{kj} \biggr|,
\]
and this is $\lesssim \epsilon j_* \rightarrow 0$ with high probability. Given $z_1,\ldots,z_n$ and $\xi$ write
\[
	m_\xi(z_1,\ldots,z_n) = \prod_{i=1}^n p_\xi^T q_i(z_i | \cdot, z_1,\ldots,z_{i-1})
\]
for the marginal density of $Z_1,\ldots,Z_n$ when $X_1,\ldots,X_n$ have distribution $p_\xi$, and similary define $m_0$ for the density of $Z_1,\ldots,Z_n$ when $X_1,\ldots,X_n$ have distribution $p_0$. Writing $M_\xi$ for the distribution associated with $m_\xi$ and $\bar{M}$ for the mixture distribution $E_\xi (M_\xi)$, we have that
\begin{align*}
	&\mathrm{KL}(M_0 \|  \bar{M}) \leq E_\xi [\mathrm{KL}( M_0 \| M_\xi) ] = E_\xi \biggl[ \int m_0(z) \log \frac{m_0(z)}{m_\xi(z)} \,dz \biggr] \\
	& = - \sum_{i=1}^n E_\xi \biggl[ \int \biggl( \prod_{i'=1}^i p_0^T q_{i'}(z_{i'} | \cdot, z_1,\ldots,z_{i'-1}) \biggr) \log \biggl(1 + \frac{(p_\xi - p_0)^T q_i(z_i | \cdot, z_1,\ldots,z_{i-1})}{p_0^T q_i(z_i | \cdot, z_1,\ldots,z_{i-1})} \biggr) \,dz_1 \ldots \, dz_i \biggr] \\
	& \leq  \sum_{i=1}^n E_\xi \biggl[ \int \biggl( \prod_{i'=1}^i p_0^T q_{i'}(z_{i'} | \cdot, z_1,\ldots,z_{i'-1}) \biggr) \biggl( \frac{(p_\xi - p_0)^T q_i(z_i | \cdot, z_1,\ldots,z_{i-1})}{p_0^T q_i(z_i | \cdot, z_1,\ldots,z_{i-1})} \biggr)^2 \,dz_1 \ldots \,dz_i \biggr] \\
	& = \epsilon^2 \sum_{i=1}^n E_\xi \biggl[ \sum_{j_1,j_2 \in B} \delta_\xi^{j_1} \mathbb{E}_{p_0} \Bigl\{ \Omega_i(Z_1,\ldots,Z_{i-1} )_{j_1j_2} \Bigr\} \delta_\xi^{j_2} \biggr] = \epsilon^2 \sum_{k_1,k_2 =1}^{j_* \wedge(d-1)} E_\xi \Bigl[ \xi_{k_1} \xi_{k_2} v_{k_1}^T \Omega v_{k_2} \Bigr] \\
	& = \epsilon^2 \sum_{k=1}^{j_* \wedge (d-1)} \lambda_k = \epsilon^2 \mathrm{tr} (\Omega) \lesssim \epsilon^2 j_* n \alpha^2.
\end{align*}
Now, as in our earlier, non-interactive, lower bound, we have
\[
	\|p_\xi - p_0 \|_1 = \epsilon \sum_{j \in B} \biggl| \sum_{k=1}^{j_* \wedge (d-1)} \xi_k v_{kj} \biggr| \gtrsim_p \epsilon j_*.
\]
We can then choose $\epsilon \asymp \min\{ (j_* n \alpha^2)^{-1/2}, p_0(j_*) / \log^{1/2} (2j_*) \}$ to prove a lower bound of
\[
	\epsilon j_* \asymp \min \Bigl\{ \Bigl( \frac{j_*}{n \alpha^2} \Bigr)^{1/2}, \frac{p_0(j_*)}{\log^{1/2} (2 j_*)} \Bigr\}.
\]
\end{proof}

\section{Examples}\label{Subsec:Examples}

\noindent {\bf Polynomially decreasing distributions. }
\label{Eg:RegVar}
Suppose that $p_0(j) \propto j^{-1-\beta}$ for some $\beta >0$. Writing $C=2 (1- 2^{-\beta})^{-1/(\beta+3/4)}$, when $n\alpha^2 \leq (d/C)^{2\beta+3/2}$, consider $j=\lceil C (n \alpha^2)^{1/(2\beta+3/2)} \rceil$. Then, when also $n \alpha^2 \geq 1$, we have that
\begin{align*}
	\sum_{\ell=j+1}^d p_0(\ell) &= \frac{\sum_{\ell=j+1}^d \ell^{-1-\beta}}{\sum_{\ell=1}^d \ell^{-1-\beta}} \leq \frac{\int_j^\infty x^{-1-\beta} \,dx}{\int_1^{d+1} x^{-1-\beta} \,dx} \leq \frac{j^{-\beta}}{1-2^{-\beta}} = \frac{j^{3/4}}{(n\alpha^2)^{1/2}} \frac{j^{-\beta-3/4}(n\alpha^2)^{1/2}}{1-2^{-\beta}} \\
	& \leq \frac{j^{3/4}}{(n \alpha^2)^{1/2}} \frac{2^{\beta+3/4}}{C^{\beta+3/4} (1-2^{-\beta})} = \frac{j^{3/4}}{(n\alpha^2)^{1/2}}.
\end{align*}
Thus, when $1 \leq n \alpha^2 \leq (d/C)^{2\beta+3/2}$ we have that $j_* \leq
\lceil C (n \alpha^2)^{1/(2\beta+3/2)} \rceil $. On the other hand, if $n\alpha^2 > (d/C)^{2\beta+3/2}$ then we will just say that $j_* \leq d$. It follows that
\[
	\mathcal{E}^\mathrm{NI}_{n,\alpha}(p_0, \mathbb{L}_1) \lesssim \frac{j_*^{3/4}}{(n\alpha^2)^{1/2}} \lesssim \min \Bigl\{ (n \alpha^2)^{-\frac{2 \beta}{4 \beta+3}}, \frac{d^{3/4}}{(n \alpha^2)^{1/2}} \Bigr\}.
\]

More generally, suppose that $p_0(j) \propto j^{-1-\beta} L(j) $ for some slowly-varying function $L:[1,\infty) \rightarrow (0,\infty)$. We recall that $L$ is said to be slowly-varying if and only if $\lim_{x \rightarrow \infty} L(tx)/L(x) =1$ for all $t>0$, and that Karamata's theorem says that
\[
	\lim_{x \rightarrow \infty} \frac{(\gamma-1) \int_x^\infty t^{-\gamma} L(t) \,dt }{ x^{-\gamma+1} L(x)} =1
\]
for any $\gamma > 1$. Writing $c_d:= \sum_{\ell=1}^d \ell^{-1-\beta} L(\ell)$, whenever $j \rightarrow \infty$ with $j \ll d$ we have that
\begin{align*}
	\sum_{\ell=j+1}^d p_0(\ell) &= c_d^{-1} \sum_{\ell=j+1}^\infty \ell^{-1-\beta}L(\ell) - c_d^{-1} \sum_{\ell=d+1}^\infty \ell^{-1-\beta} L(\ell) \sim c_d^{-1} \sum_{\ell=j+1}^\infty \ell^{-1-\beta}L(\ell) \\
	& \sim \frac{j^{-\beta} L(j)}{c_d \beta} = \frac{jp_0(j)}{\beta}.
\end{align*}
Letting $x_{n \alpha^2}:= \inf \{x \geq  1 : L(x) < \frac{x^{3/4 + \beta}}{(n \alpha^2)^{1/2}} \}$, we can see that
\[
	\mathcal{E}^\mathrm{NI}_{n,\alpha}(p_0, \mathbb{L}_1) \lesssim \frac{ \min(x_{n\alpha^2},d)^{3/4}}{ (n \alpha^2)^{1/2}}.
\]

\bigskip

Let us discuss the lower bounds. Writing $c=\frac{\beta^2(2 \beta + 3/2)}{2(1-2^{-\beta})^2}$ and $j=\lfloor \{c n \alpha^2/\log(n \alpha^2)\}^{1/(2\beta+3/2)} \rfloor$, when $\log(n \alpha^2) \geq \log c + (2 \beta+3/2) \log 2$ and $\frac{c n \alpha^2}{\log(n \alpha^2)} \leq d^{2 \beta + 3/2}$, we have that
\begin{align*}
	\frac{j p_0(j)}{ \log^{1/2}(2j)} &= \frac{j^{-\beta}}{\log^{1/2}(2j) \sum_{\ell=1}^d \ell^{-1-\beta}} \geq \frac{\beta j^{-\beta}}{\log^{1/2}(2j) (1-2^{-\beta})} = \frac{j^{3/4}}{(n \alpha^2)^{1/2}} \frac{\beta}{1-2^{-\beta}} \frac{(n \alpha^2)^{1/2}}{ \log^{1/2}(2j) j^{\beta+3/4} } \\
	&\geq \frac{j^{3/4}}{(n \alpha^2)^{1/2}} \frac{\beta (2 \beta+3/2)^{1/2}}{c^{1/2} (1-2^{-\beta})} \biggl\{ 1 + \frac{\log c + (2 \beta+3/2) \log 2}{\log ( n \alpha^2)} \biggr\}^{-1/2} \geq \frac{j^{3/4}}{(n \alpha^2)^{1/2}}.
\end{align*}
Hence, we have $\ell_* \geq j$. On the other hand, when $ \frac{cn \alpha^2}{\log (n \alpha^2)} > d^{2 \beta + 3/2}$ and $\log(n \alpha^2) > c 2^{2\beta+3/2}$, we have
\begin{align*}
	\frac{dp_0(d)}{\log^{1/2}(2d)} \geq \frac{d^{3/4}}{(n \alpha^2)^{1/2}} \frac{\beta}{1-2^{-\beta}} \frac{(n \alpha^2)^{1/2}}{d^{3/4+\beta} \log^{1/2}(2d)} \geq \frac{d^{3/4}}{(n \alpha^2)^{1/2}},
\end{align*}
and so $\ell_* =d$. In either case, then,
\begin{eqnarray*}
	\mathcal{E}_{n,\alpha}^\mathrm{NI}(p_0, \mathbb{L}_1)& \gtrsim &\frac{\{(n\alpha^2) / \log(n \alpha^2) \}^{(3/4)/(2\beta+3/2)} \wedge d^{3/4}}{(n \alpha^2)^{1/2}} \\
&= &\bigl\{ n \alpha^2 \log^{3/(4 \beta)}(n \alpha^2) \bigr\}^{-2\beta/(4 \beta+3)} \wedge \frac{d^{3/4}}{(n \alpha^2)^{1/2}}.
\end{eqnarray*}

More generally, suppose that $p_0(j) \propto j^{-1-\beta} L(j)$ and recall the definition of $x_{n \alpha^2}$ from Example~\ref{Eg:RegVar}. Taking $j=\min(\lfloor x_{n\alpha^2 / \log(n \alpha^2)} \rfloor , d)$ in Theorem~\ref{Thm:NonIntLowerBound}, we have that
\[
	\mathcal{E}_{n,\alpha}^\mathrm{NI}(p_0, \mathbb{L}_1) \gtrsim \frac{\min( x_{n\alpha^2 / \log(n \alpha^2)}, d)^{3/4}}{(n \alpha^2)^{1/2}},
\]
which matches our upper bound up to a log factor.

\bigskip

\noindent {\bf Exponentially decreasing distributions. } Suppose that $p_0(j) \propto \exp(-j^\beta)$ for some $\beta >0$. Writing $C$ for a large constant, if $(\frac{1}{4 \beta} + \frac{1}{2}) \log(C n \alpha^2) \leq d^\beta$ then consider $j = \lceil \{ \log( C n \alpha^2)/2 - (1-1/(4\beta)) \log \log(C n\alpha^2) \}^{1/\beta} \rceil$. Then
\begin{align*}
	\sum_{\ell=j}^d p_0(\ell) \leq \frac{\int_j^\infty \exp(- x^\beta) \,dx}{\int_1^{d+1} \exp( - x^\beta) \,dx}  \lesssim j^{1-\beta} e^{-j^\beta} \lesssim \frac{ \log^{3/(4\beta) } (C n \alpha^2) }{ \sqrt{Cn \alpha^2} } \lesssim \frac{j^{3/4}}{ \sqrt{C n \alpha^2} },
\end{align*}
and we can therefore see that $j_* \lesssim \log^{1/\beta} (n \alpha^2)$. As a result,
\[
	\mathcal{E}^\mathrm{NI}_{n,\alpha}(p_0, \mathbb{L}_1) \lesssim \min \biggl\{ \frac{\log^{3/(4 \beta)}( n \alpha^2)}{(n \alpha^2)^{1/2}}, \frac{d^{3/4}}{(n \alpha^2)^{1/2}} \biggr\}.
\]

Concerning the lower bounds, write $c$ for a small constant and consider $j = \lfloor \{ \log( c n \alpha^2)/2 + \log \log(c n\alpha^2)/(4\beta) - \log \log \log(c n \alpha^2) /2 \}^{1/\beta} \rfloor$. If $j \leq d$ then we have
\begin{align*}
	\frac{jp_0(j)}{\log^{1/2}(2j)} \gtrsim \frac{\log^{1/\beta}(c n \alpha^2) e^{-j^\beta}}{\log^{1/2}( \log(c n \alpha^2) ) } \gtrsim \frac{\log^{3/(4 \beta)}(c n \alpha^2 )}{ \sqrt{c n \alpha^2}} \gtrsim \frac{j^{3/4}}{ \sqrt{c n \alpha^2}}.
\end{align*}
If, on the other hand, $j > d$, then
\begin{align*}
	\frac{dp_0(d)}{ \log^{1/2}(2d)} \gtrsim \frac{d \exp(-d^\beta)}{ \log^{1/2}(2d)} \gtrsim \frac{\log^{3/(4 \beta)}( cn \alpha^2)}{ \sqrt{c n \alpha^2}} \gtrsim \frac{d^{3/4}}{ \sqrt{c n \alpha^2}}.
\end{align*}
In either case, then, we have
\[
	\mathcal{E}_{n,\alpha}^\mathrm{NI}(p_0 , \mathbb{L}_1) \gtrsim \min \biggl\{ \frac{ \log^{3/(4 \beta)}(n \alpha^2)}{\sqrt{n \alpha^2}}, \frac{d^{3/4}}{\sqrt{n \alpha^2}} \biggr\},
\]
and this matches our previous upper bound.

\bibliographystyle{plainnat}

\end{document}